\documentclass{amsart} 
\usepackage{amssymb} 
\usepackage{mathrsfs} 
\usepackage{xy} 
\usepackage{multirow}
\usepackage{adjustbox}
\usepackage{filecontents} 
\usepackage[backref=page]{hyperref}
\hypersetup{
    colorlinks=true,
    linkcolor=blue,
    citecolor=red, 
    }

\usepackage{cleveref} 
\numberwithin{equation}{section}
\theoremstyle{plain}
\newtheorem{thm}[equation]{Theorem}
\newtheorem{prop}[equation]{Proposition}
\newtheorem{lem}[equation]{Lemma}
\newtheorem{cor}[equation]{Corollary}
\theoremstyle{definition}

\newtheorem{conj}[equation]{Conjecture}

\newtheorem{as}[equation]{Assumption}

\newtheorem{ww}[equation]{}
\theoremstyle{remark}

\newtheorem*{rem*}{Remark}

\DeclareMathOperator{\Br}{Br}

\DeclareMathOperator{\Char}{char}

\DeclareMathOperator{\disc}{disc}

\DeclareMathOperator{\Id}{Id}
\DeclareMathOperator{\Frac}{Frac}
\DeclareMathOperator{\Herm}{Herm}

\DeclareMathOperator{\End}{End}
\DeclareMathOperator{\Int}{Int}

\DeclareMathOperator{\Per}{per}
\DeclareMathOperator{\Ind}{ind}
\DeclareMathOperator{\Op}{op}

\DeclareMathOperator{\Rad}{Rad}
\DeclareMathOperator{\Rank}{Rank}
\DeclareMathOperator{\Spec}{Spec}
\DeclareMathOperator{\SU}{SU}
\DeclareMathOperator{\SO}{SO}

\DeclareMathOperator{\U}{U}
\DeclareMathOperator{\Nrd}{Nrd}

\DeclareMathOperator{\PGL}{PGL}
\DeclareMathOperator{\Proj}{Proj}

\DeclareMathOperator{\ad}{ad} 

\DeclareMathOperator{\rdim}{rdim}

\newcommand{\lBr}{{_l}\mathrm{Br}}

\title{Hasse principle for hermitian spaces over semi-global fields}
\author{Zhengyao Wu}
\address{
Zhengyao Wu\\
Dept. of Mathematics and Computer Science\\
Emory University\\
400 Dowman Drive, W401\\
Atlanta, GA 30322}
\email{zwu22@mathcs.emory.edu}
\begin{document}
\begin{abstract} 
In a recent paper, Colliot-Th\'el\`ene, Parimala and Suresh conjectured that a local-global principle holds for projective homogeneous spaces under connected linear algebraic groups over function fields of p-adic curves. 
In this paper, we show that the conjecture is true for any linear algebraic group whose almost simple factors
of its semisimple part are isogenous to unitary groups or special unitary groups of hermitian or skew-hermitian spaces over central simple algebras with involutions. 
The proof implements patching techniques of Harbater, Hartmann and Krashen. 
As an application, we obtain a Springer-type theorem for isotropy of hermitian spaces over 
odd degree extensions of function fields of p-adic curves. 
\end{abstract}
\maketitle
\setcounter{tocdepth}{1} 
\tableofcontents
\section{Introduction}\label{sec1}
Let $T$ be a complete discrete valuation ring with residue field $k$. Let $K$ be the field of fractions of $T$. 
Let $F$ be the function field of a smooth, projective, geometrically integral curve $\mathscr X_0$ over $K$. 
Recently, such a field $F$ has been called a \textit{semi-global} field.
Let $\Omega$ be the set of all rank one discrete valuations on $F$ (or the set of all divisorial discrete valuations from all codimension one points of all regular projective models $\mathscr X\to\Spec(T)$ of the curve $\mathscr X_0$). 
For each $v\in \Omega$, let $F_v$ be the completion of $F$ at $v$. 
Let $G$ be a connected linear algebraic group over $F$ and let $X$ be a projective homogeneous space 
 under $G$ over $F$. 
We say that the \textit{Hasse principle holds} for $X$ if
\[\prod\limits_{v\in\Omega}X(F_v)\ne\emptyset \implies X(F)\ne\emptyset. \]

Colliot-Th\'el\`ene, Parimala and Suresh \cite[3.1]{CTPS} have proved that if $\Char k\ne 2$ and $q:V\to F$ is a quadratic form with $\dim_F(V)\ge 3$, then the Hasse principle holds for every projective homogeneous space under $\SO(q)$. 
Reddy and Suresh \cite[2.6]{RS} have proved that if $A$ is a central simple $F$-algebra of degree coprime to $\Char k$, then the Hasse principle holds for every projective homogeneous space under $\PGL_1(A)$. 
After \cite[3.1]{CTOP} and \cite[5.7]{CTGP}, Harbater, Hartmann and Krashen \cite[9.2]{HHK2} have proved that if $k$ is algebraically closed and $\Char k=0$, then the Hasse principle holds for projective homogeneous spaces under connected \textit{rational} groups. 

In this article we explore the Hasse principle for projective homogeneous spaces under $G$ over $F$ for  
certain groups of classical types.

\begin{as}\label{as1}
Suppose $K$, $k$, $F$, $\Omega$, $G$, $X$ are as before and $\Char k\ne 2$.  
Let $A$ be a finite-dimensional simple associative $F$-algebra with an involution $\sigma$ such that $F=Z(A)^{\sigma}$.  
Let $h: V\times V\to A$ be an $\varepsilon$-hermitian space over $(A, \sigma)$ for $\varepsilon\in\{1, -1\}$. 
Let 
\[G=\left\{
\begin{array}{ll}
\SU(A,\sigma, h), & \text{if } \sigma \text{ is of the first kind;} \\
\U(A, \sigma, h), & \text{if } \sigma \text{ is of the second kind.}\\
\end{array}
\right.\]
%
\end{as}

In \cref{sec7} and \cref{sec9}, we will prove the following Hasse principle for discrete valuations based on the Hasse principle \cite[3.7]{HHK1} with respect to  patching. 

\begin{thm}[Main]\label{main-thm}
Under \cref{as1}, suppose that at least one of the following is satisfied. 
 
(1) $\Ind(A)\le 2$;

(2) $\Per(A)=2$, $|l^*/{l^*}^2|\le 2$ and ${_2}\Br(l)=0$ for all finite extensions $l/k$.  

Then the Hasse principle holds for $X$. 
\end{thm}

\begin{rem*}
In case (1), the underlying division algebra of $A$ is $F$, or a quadratic field extension of $F$, or a quaternion division algebra with center $F$, or a quaternion division algebra whose center is a quadratic extension of $F$. 

In case (2), if $\sigma$ is of the first kind, then $\Per(A)=2$ since $A\simeq A^{\Op}$; 
if $\sigma$ is of the second kind, in general we do not have $\Per(A)=2$. 
By \cite[XIII, \S 2]{Ser}, examples of such $k$ in (2) are finite fields or fields of Laurent series with coefficients in an algebraically closed field of characteristic $0$, for example $\mathbb C((t))$. 
\end{rem*}
%

Colliot-Th\'el\`ene, Parimala and Suresh  \cite[conj.~1]{CTPS} made the following

\begin{conj} Let $K$ be a $p$-adic field and $F$ a function field of a curve over $K$.
 Let $G$ be a  connected linear algebraic group over $F$ and let $X$ be a projective homogeneous space under $G$ over $F$. Then the Hasse principle holds for $X$.  
\end{conj}

Summarizing \cite[2.6]{RS}, \cite[3.1]{CTPS} and \cref{main-thm}, we have 
the following:   
\begin{cor}\label{main-cor}
Under \cref{as1}, let $K$ be a $p$-adic field. 
Then  
the conjecture of Colliot-Th\'el\`ene, Parimala and Suresh is true for $X$ under $G$ such that there exists an isogeny from a product of almost simple groups of one of the following types to the semisimple group $G/\Rad(G)$. $${^1}A_n, \quad {^2}A_n^*,\quad  B_n, \quad C_n, \quad D_n~(D_4 \text{ nontrialitarian}),$$ where ${^2}A_n^*$ means that the almost simple factor is isogenous to a unitary group $\U(A, \sigma, h)$ such that $\sigma$ is of the second kind and $\Per(A)=2$. 
\end{cor}

Let $F$ be a field of characteristic not 2. 
Let $q$ be a quadratic form over $F$. 
Springer \cite{Spr1} proved that 
if $q$ is isotropic over an odd degree extension of $F$, then $q$ is isotropic over $F$.
Let $A$ be a central simple algebra over $F$ with an involution $\sigma$.
Let $h: V\times V\to A$ be an $\varepsilon$-hermitian form over $(A, \sigma)$ for $\varepsilon \in\{1,-1\}$.
Let $M$ be an odd degree extension of $F$. It is natural to ask whether the isotropy of 
$h_M$ implies the isotropy of $h$. This question has been studied by many mathematicians and they have obtained partial answers. 
Bayer-Fluckiger and Lenstra \cite{BL} proved that if $h_M$ is hyperbolic, then $h$ is hyperbolic. 
Lewis \cite{Lewis} and Barqu\'ero-Salavert \cite{BS} proved that if $h_1$ and $h_2$ are two $\varepsilon$-hermitian spaces such that $(h_1)_M\simeq (h_2)_M$, then $h_1\simeq h_2$. 
Parimala, Sridharan and Suresh \cite{PSS} proved that it is true if $A$ is a quaternion algebra and $\sigma$ is of the first kind; they also provided an example to show that this is not true in general if $\Ind(A)$ is odd and $\sigma$ of the second kind. 
Let $E=\End_A(V)$ and let $\tau$ be the adjoint involution of $h$. 
Black and Qu\'eguiner-Mathieu \cite{BQ} proved that it is true when $\deg E=12$ and $\tau$ is orthogonal; and it is also true when $\deg E=6$, $\Per E=2$ and $\tau$ is unitary. 
In \cref{sec10}, we will prove the following: 

\begin{thm}\label{odd}
Let $p$ be an odd prime.
Let $K$ be a $p$-adic field. 
Let $F$ be the function field of a curve over $K$.
Let $\Omega$ be the set of all rank one discrete valuations on $F$. 
Let $A$ be a finite-dimensional central simple $F$-algebra with an involution $\sigma$ of the first kind.  
Let $h$ be an $\varepsilon$-hermitian space over $(A, \sigma)$ for $\varepsilon\in\{1, -1\}$. 
Let $M$ be an odd degree extension of $F$. 
If $h_M$ is isotropic, then $h$ is isotropic. 
\end{thm}

\textbf{Acknowledgements.}

The author thanks his advisor Professor V.~Suresh for thorough detailed instructions and Professor R.~Parimala for helpful discussions. 
The author also thanks Emory university, where he studies mathematics. 

\section{Preliminaries}\label{prelim}
\subsection{Projective homogeneous spaces}\label{sec1.5}
Let $F$ be an arbitrary field, $\Char(F)\ne 2$. 
Let $A$ be a central simple algebra whose center $Z(A)$ is a field extension of $F$. 
Let $\sigma$ be an involution on $A$ such that $Z(A)^{\sigma}=F$. 
Let $V$ be a finitely generated right $A$-module and let $h: V\times V\to A$ be an $\varepsilon$-hermitan form over $(A, \sigma)$ for $\varepsilon\in \{1, -1\}$.  
A submodule $W$ of $V$ is {\it totally isotropic} if $h(x, x) = 0$ for all $x\in W$. 

If 
\[G=G(A,\sigma,h)=\left\{
\begin{array}{ll}
\SU(A,\sigma, h) & \text{if } \sigma \text{ is of the first kind;} \\
\U(A, \sigma, h) & \text{if } \sigma \text{ is of the second kind,}\\
\end{array}
\right.\]
then $G$ is a connected \textit{rational} linear algebraic group over $F$ by Cayley parametrization (see \cite[Lem.~5]{CP} or \cite[p.~195, Lem.~1]{Mer}). 
Let $X$ be a projective homogeneous space under $G$ over $F$. 
In this section, we describe $X$. 

By Wedderburn's theorem,   
$A = M_m(D)$ for a central division algebra $D$ over $F$. 
Let $W$ be a finitely generated right $A$-module, 
Then $W\simeq (D^m)^s$ for an integer $s\ge 0$. 
Then $\dim_F(W) = sm\dim_F(D) = s\deg(A)\Ind(A)$. 
The reduced dimension \cite[1.9]{inv} of $W$ over $A$ is defined to be $\rdim_A(W)=\dim_F(W)/\deg(A)=s\Ind(A)$. 
 
Let $n = \Rank_F(G)$. Then 
$$
\rdim(V)=\left\{
\begin{array}{ll}
n+1, & \text{if }\sigma\text{ is unitary;}\\
2n+1, & \text{if }A=F, \sigma=\Id_F\text{ and }\dim_F(V)\text{ is odd;}\\
2n, & \text{otherwise.}\\
\end{array}
\right.$$

Let $0 < n_1 < \cdots < n_r \le n$ be an increasing sequence of integers. 
For every field extension $L/F$, let 
$$
\begin{array}{rl}
X(n_1, \cdots , n_r)(L) & =~\{(W_1, \cdots, W_r)~|~0\subsetneq W_1\subsetneq \cdots\subsetneq W_r,
 ~W_i \text{ is a totally}  \\
& \text{isotropic subspace of } V_L, ~\rdim_{A_L} W_i=n_i \text{ for all }1\le i\le r \}.
\end{array}
$$
Alternatively, by \cite[6.2]{inv} and \cite[16.4]{Kar}, 
$$
\begin{array}{ll}
& X(n_1, \cdots , n_r)(L)= \{(I_1, \cdots, I_r)~|~0\subsetneq I_1\subsetneq \cdots\subsetneq I_r,
~I_j \text{ is a totally  isotropic }  \\
& \text{ideal of } \End_{A_L}(V_L), ~\rdim_{A_L} I_j=n_j \text{ for all }1\le j\le r \}.
\end{array}
$$
When $r = 1$, we denote $X(n_1)$ by $X_{n_1}$. 
Let $\ad_h$ denote the adjoint involution of $h$ on $\End_A(V)$. 
By \cite{MPW1}, if $\ad_h$ is orthogonal, $\disc(h)=1$, $r=1$ and $n_1=n$, then $X_n$ has two connected components $X_n^+$ and $X_n^-$. 
In this case, for $\varepsilon\in\{+,-\}$, denote
\begin{equation}\label{pm}
X^{\varepsilon}(n_1, \cdots, n_r)(L)
= \{(I_1, \cdots, I_r)\in X(n_1, \cdots, n_r)(L)~|~
I_r\in X_n^{\varepsilon}(L)\},
\end{equation}

By \cite[sec.~5 and sec.~9]{MPW1, MPW2}, the following $X$ is a projective homogeneous space under $G$ and
 any projective homogeneous space under $G$ has this form. 

\begin{equation}\label{phs}
X=
\left\{
\begin{array}{ll}
X(n_1,\cdots, n_r), n_r<\lfloor n/2\rfloor, & \text{ if }\sigma\text{ is unitary};\\
X(n_1,\cdots, n_r), & \text{ if }A=F, \sigma=\Id_F\text{ and }\dim_F(V)\text{ is odd};\\
X(n_1,\cdots, n_r), & \text{ if }\ad_h \text{ is symplectic};\\
X(n_1,\cdots, n_r), n_r<n, & \text{ if }\ad_h \text{ is orthogonal and }\disc(h)\ne 1;\\
\left\{
\begin{array}{l}
X(n_1,\cdots, n_r), n_r<n \text{ or }\\
X^{\pm}(n_1,\cdots, n_r), \\
~n_{r-1}<n-1 (\text{if }r>1), n_r=n\\
\end{array}
\right.,
& \text{ if }\ad_h \text{ is orthogonal and } \disc(h)=1.\\
\end{array}
\right.
\end{equation}


\begin{lem}[{\cite[sec.~5 and sec.~9]{MPW1, MPW2}}]\label{2.1}
	Let $0 < n_1 <  \cdots <  n_r \le n$, $\varepsilon \in \{+, -\}$.  
	Let $L/F$ be a field extension. Then  
	
	(1) $X(n_1, \cdots , n_r)(L) \ne \emptyset$ if and only if $X_{n_r}(L) \ne \emptyset$ and $\Ind(A_L)|\gcd\{n_1, \cdots, n_r\}$. 
	
	(2) $X^{\varepsilon}(n_1, \cdots , n_r)(L) \ne \emptyset$ if and only if $X_{n_r}^{\varepsilon}(L) \ne \emptyset$ and $\Ind(A_L)|\gcd\{n_1, \cdots, n_r\}$. 
\end{lem}

%

\subsection{Morita equivalence}

Since $A = M_m(D)$ for a central division algebra $D$ over $F$ and $\sigma$ is an involution on $A$, by \cite[3.1, 3.11, 3.20]{inv}, $D$ also has an involution $\tau$ of same kind as $\sigma$. 
By Morita equivalence \cite[ch.~I, 9.3.5]{Knus}, there exists an $\varepsilon' $-hermitian form $(V_0, h_0)$ over $(D, \tau)$ for $\varepsilon' \in \{ 1, -1 \}$. 
By the definition of the reduced dimension we have 
$\rdim_A(V) = \rdim_D(V_0)$. 
For $0< n_1< \cdots < n_r\le n$, let $X$ be the 
projective homogeneous space under $G(A,\sigma,h)$ and $X_0$ be the projective homogeneous space under $G(D,\tau,h_0)$. 


\begin{lem}[{\cite[16.10]{Kar}}]\label{2.2}
$X(n_1, \cdots , n_r)\simeq X_0(n_1, \cdots , n_r)$.  
\end{lem}

Actually, we only need $X(n_1, \cdots , n_r)(L)\ne\emptyset\iff X_0(n_1, \cdots , n_r)(L)\ne\emptyset$. 
This is true since Morita equivalence preserves isotropy \cite[ch.~I, 9.3.5]{Knus} and it preserves reduced dimension. 


\begin{lem}\label{2.4}
Suppose $\rdim(V)=2n$, $\ad_h$ is orthogonal, $\disc(h)=1$, $n_{r-1}<n-1$ (if $r>1$) and $n_r=n$. 
If $\Ind(A_L)|\gcd\{n_1, \cdots, n_r\}$, then $X^{\varepsilon}(n_1, \cdots , n_r)(L)\ne\emptyset$ if and only if $X_0^{\varepsilon}(n_1, \cdots , n_r)(L)\ne\emptyset$, for $\varepsilon\in\{+,-\}$. 
\end{lem}

\begin{proof}
	By \cref{2.1} and \cref{2.2}, it suffices to show that for $\varepsilon\in\{+,-\}$,  $$X_n^{\varepsilon}(L)\ne\emptyset\iff(X_0)_n^{\varepsilon}(L)\ne\emptyset.$$
	This is true by the definition of $X^{\varepsilon}_n$ (see the paragraph at \cite[p.577, 5.41, 5.42]{MPW1}).  
\end{proof}

\begin{lem}\label{pmkey}
Suppose $\rdim(V)=2n$, $\ad_h$ is orthogonal, $\disc(h)=1$, $n_{r-1}<n-1$ (if $r>1$) and $n_r=n$. 
Let $X^{\varepsilon}=X^{\varepsilon}(n_1, \cdots, n_r)$ for $\varepsilon\in\{+,-\}$. 
Then 
$X^+(L)\ne\emptyset$ and $X^-(L)\ne\emptyset$ if and only if $A_L$ is split and $h_L$ is hyperbolic. 
\end{lem}

\begin{proof}
Suppose that $A_L$ is split and $h_L$ is hyperbolic. 
Then $h_L$ is Morita equivalent to a hyperbolic quadratic form $q$ over $L$. 
Let $X_0^{\pm}$ be corresponding projective homogeneous spaces under $\SO_{2n}(q)$. 
Since the Witt index of $q$ is $n$, we have $(X_0)_n^{+}(L)\ne\emptyset$ and $(X_0)_n^{-}(L)\ne\emptyset$. 
Since $A_L$ is split, we have $\Ind(A_L)=1|\gcd\{n_1,\cdots, n_r\}$. 
By \cref{2.1}(2), $X_0^{+}(L)\ne\emptyset$ and $X_0^{-}(L)\ne\emptyset$. 
By \cref{2.4}, $X^+(L)\ne\emptyset$ and $X^-(L)\ne\emptyset$. 

Conversely, suppose $X^+(L)\ne\emptyset$ and $X^-(L)\ne\emptyset$. 
Let $W^+\in X^+(L)$ and  $W^-\in X^-(L)$.
Since there exists a totally isotropic subspace of $h_L$ of reduced dimension $n$, which is equal to the Witt index of $h_L$, we have that $h_L$ is hyperbolic. 
By Witt's extension theorem \cite[\S~4, no.~3, th.~1]{Balg9} there exists   
$\varphi\in \U(A, \sigma, h)$ such that $\varphi(W^+)=W^-$.
Since $ \SU(A, \sigma, h)$ sends  $X^+(L)$ into $X^+(L)$ and 
$X^-(L)$ into $X^-(L)$, we obtain $\varphi \not\in \SU(A, \sigma, h)$. 
Thus, by \cite[2.6, lem.~1.~a)]{Kneser}, $A_L$ is split. 
\end{proof}

\subsection{Ramification and patching data}\label{data}
Let $\mathscr X$ be a regular integral scheme with function field $F$. 
For every codimension one point $x$ of $\mathscr X$, let $k(x)$ denote the residue field at $x$. 
Let $l$ be a prime which is unit at every point of $\mathscr X$. 
Let $\lBr(F)$ be the $l$-torsion subgroup of the Brauer group $\Br(F)$ of $F$. 
Then there is a residue homomorphism $\partial_x: \lBr(F) \to H^1(k(x), \mathbb Z/l\mathbb Z)$. 
We say that an element $\alpha \in \lBr(F)$ is {\it ramified} at $x$ if $\partial_x(\alpha)\ne 0$; we say that $\alpha$ is {\it unramified} at $x$ if $\partial_x(\alpha) = 0$. 
The {\it ramification divisor} of $\alpha$
is defined as $\sum x$, where $x$ runs over all codimension one points of $\mathscr X$ with $\partial_x(\alpha) \ne 0$. 
 
Let $K$ be a complete discretely valued field with ring of integers $T$ with a parameter $t$ and residue field $k$.
Let $F$ be the function field of a curve over $K$. 
Let $\mathscr X$ be a regular proper two dimensional scheme over $\Spec(T)$ with function field $F$ and $\mathscr X_1$ its special fiber.
Let $l$ be a prime not equal to $\Char k$. 
Let $\alpha \in \lBr(F)$. 
Then there exists a
regular proper model $\mathscr X$ of $F$ with the support of the ramification divisor of $\alpha$ and $\mathscr X_1$ is a union of regular curves with only normal crossings (see \cite{Lip75} or \cite{Abh}). 
For the patching data, we adopt notations similar to \cite[3.3]{HHK1}. 
For every closed point $P$ of $\mathscr X$, let $\widehat{R_P}$ be the completion of the local ring $R_P$ of $\mathscr X$ at $P$ and $F_P=\Frac(\widehat{R_P})$. 
Let $\mathscr X_\eta$ be an irreducible component of $\mathscr X$ and $U$ be a non-empty open subset of $\mathscr X_\eta$ containing only smooth points. 
Let $R_U$ be the set of elements in $F$ 
which are regular at every closed point of $U$. 
Let $\widehat{R}_U$ be the $(t)$-adic completion of $R_U$ and $F_U=\Frac(\widehat{R}_U)$. 
Let $\mathcal P$ be a finite set of closed points of $\mathscr X$ of our choice and let $\mathcal U$ be the set of connected components of the complement of $\mathcal P$ in the special fiber of $\mathscr X$. 
Harbater, Hartmann and Krashen \cite[3.7]{HHK1} have proved that  if $G$ is a rational connected linear algebraic group over $F$ and $X$ is a projective homogeneous space under $G$, then 
\[\prod_{P\in\mathcal P}X(F_P)\times\prod_{U\in\mathcal U}X(F_U)\ne\emptyset\implies X(F)\ne\emptyset.\]

\subsection{Complete discrete valued field}\label{sec2}
In this section we recall a theorem of Larmour on Hermitian spaces over discretely valued fields and prove results concerning  maximal orders.  
 
Let $(K,v)$ be a discrete valued field with valuation ring $R_v$, maximal ideal $m_v$ and residue field $k(v)=R_v/m_v$, $\Char(k(v))\ne 2$.  
Let $(\widehat{R_v}, \widehat{m_v})$ be the completion of $(R_v, m_v)$ and $K_v=\Frac(\widehat{R_v})$. 
Let $\widehat{v}$ be the extension of $v$ to $K_v$. 
We have $k(\widehat{v})=\widehat{R_v}/\widehat{m_v}=k(v)$. 
Let $D$ be a finite-dimensional division algebra over $K$ with an involution $\sigma$ such that $Z(D)^{\sigma}=K$. 
\textit{Suppose that $D\otimes_KK_v$ is a division algebra over $K_v$}. 
By \cite[ch.~II, 10.1]{CF}, $\widehat v$ extends to a valuation $v'$ on $Z(D\otimes_KK_v)$ such that $v'(x)=\frac{1}{[Z(D\otimes_KK_v):K_v]}v(N_{Z(D\otimes_KK_v)/K_v}(x))$ for all $x\in Z(D\otimes_KK_v)^*$.
By \cite{W1}, $v'$ extends to a valuation $w$ on $D\otimes_KK_v$ such that $w(x)=\frac{1}{\Ind(D\otimes_KK_v)}v'(\Nrd_{D\otimes_KK_v/Z(D\otimes_KK_v)}(x))$ for all $x\in (D\otimes_KK_v)^*$. 
The restriction of $w$ to $D$ is a valuation on $D$ and $w(x)=\frac{1}{\Ind(D)}v(\Nrd_{D/K}(x))$ for all $x\in D^*$. Since $\Nrd_{D/K}(x)=\Nrd_{D/K}(\sigma(x))$, we have $w(\sigma(x))=w(x)$ for all $x\in D$. 
Since $\Nrd_{D/K}(x)=\Nrd_{D/K}(\sigma(x))$, we have $w(\sigma(x))=w(x)$ for all $x\in D$. 
Let $t_D$ be the \textit{parameter} of $(D, w)$ (see \cite[13.2]{Rei}). 
We may choose $\pi_D\in D^*$ such that 
$w(\pi_D)\equiv w(t_D)\mod{2w(D^*)}$ and $\sigma(\pi_D)=\pm \pi_D$ (see \cite[2.7]{L2}). 
Let $R_w=\{x\in D~|~w(x)\ge 0\}$ and $\mathfrak m_w=\{x\in D~|~w(x)> 0\}$. 
Let $D(w)=R_w/\mathfrak m_w$ be the residue division algebra (see \cite[13.2]{Rei}) of $(D, w)$ over $k(v)$ with involution $\sigma_w$ such that $\sigma_w(q_w(x))=q_w(\sigma(x))$ for all $x\in R_w$, where $q_w(x)=x+\mathfrak m_w$. 

\begin{lem}\label{unique-max}
Suppose that $D\otimes_KK_v$ is a division algebra over $K_v$. 
There exists a unique maximal $R_v$-order $\Lambda$ in $D$ and the following four sets are identical. 
\begin{enumerate}
\item the maximal $R_v$-order $\Lambda$ in $D$; 
\item the valuation ring $R_w=\{x\in D~|~w(x)\ge 0\}$; 
\item $N=\{x\in D~|~N_{D/K}(x)\in R_v\}$; 
\item the integral closure $S$ of $R_v$ in $D$.
\end{enumerate}
\end{lem}

\begin{proof}
\textit{Existence}: By \cite[10.4]{Rei}, there exists a maximal $R_v$-order $\Lambda$ in $D$. 

\textit{Uniqueness}: If $\Lambda$ and $\Lambda'$ are two maximal $R_v$-orders in $D$, by \cite[11.5]{Rei} $\Lambda\otimes\widehat{R_v}$ and $\Lambda'\otimes\widehat{R_v}$ are two maximal $\widehat{R_v}$-orders in $D\otimes K_v$. 
By \cite[12.8]{Rei}, the maximal $\widehat{R_v}$-order in $D\otimes K_v$ is unique. 
Then $\Lambda\otimes\widehat{R_v}=\Lambda'\otimes\widehat{R_v}$. 
Then by \cite[5.2]{Rei}, $\Lambda=(\Lambda\otimes\widehat{R_v})\cap D=(\Lambda'\otimes\widehat{R_v})\cap D=\Lambda'$. 

\textit{Equalities}: Let $\Lambda$ be the unique maximal $R_v$-order in $D$. 
By \cite[12.7, 12.8]{Rei}, the following sets are equal 
\begin{itemize}
\item the maximal $\widehat{R_v}$-order $\widehat{\Lambda}=\Lambda\otimes \widehat{R_v}$ in $D\otimes K_v$; 
\item the valuation ring $\widehat{R_w}=\{x\in D\otimes K_v~|~w(x)\ge 0\}$; 
\item $\widehat{N}=\{x\in D\otimes K_v~|~N_{D\otimes K_v/K_v}(x)\in \widehat{R_v}\}$; 
\item the integral closure $\widehat{S}$ of $\widehat{R_v}$ in $D\otimes K_v$.
\end{itemize}

The proof of (1) equals (2): For $x\in D$, $w(x\otimes 1)=w(x)$, then $\widehat{R_w}\cap D=R_w$. 
Then $\Lambda=\widehat{\Lambda}\cap D=\widehat{R_w}\cap D=R_w$. 

The proof of (1) equals (3): For $x\in D$, by \cite[\S~17, no.~3, prop.~4, (30)]{Bcomm89}, $\Nrd_{(D\otimes K_v)/K_v}(x\otimes 1)=\Nrd_{D/K}(x)$, then $\widehat{N}\cap D=N$. 
Then $\Lambda=\widehat{\Lambda}\cap D=\widehat{N}\cap D=N$. 

The proof of (1) equals (4): By \cite[8.6]{Rei}, $\Lambda\subset S$. 
Also, $S\subset \widehat{S}\cap D=\widehat{\Lambda}\cap D=\Lambda$. 
Therefore $\Lambda=S$. 
\end{proof}

The next lemma will be applied in \cref{qua-order}. 

\begin{lem}\label{ord}
Suppose $D=(a,b)$ is a quaternion division algebra given by $i^2=a$, $j^2=b$, $ij=-ji$, where $a,b\in K$. 
Suppose $D\otimes_KK_v$ is a division algebra over $K_v$. 
If $v(a) = 0$ and   $v(b)\in \{0,1\}$, then $\Lambda=   R_v +   R_vi +  R_vj +  R_vij$ is the unique maximal $  R_v$-order in $D$. 
\end{lem}

\begin{proof}   
  By \cref{unique-max},  $\Lambda$ is the unique maximal order if and only if $\Lambda$ is the 
 integral closure of $R_v$ in $D$.
 Since $i$ and $j$ are integral over $R_v$, every element of 
$\Lambda$ is integral over $R_v$.  

Let $x\in D$. Then 
$$x= y(x_0+x_1i+x_2j+x_3ij)$$
for some $y \in K^*$ and   $ x_0, x_1, x_2, x_3  \in R_v $  with  $\min\limits_{0\le l\le 3}\{ v(x_l)\}=0$ (i.e. $(\overline{x_0}, \overline{x_1}, \overline{x_2}, \overline{x_3})\ne \vec{0}$ in $k(v)^4$).

Suppose that $x$  is integral over $ R_v$.  
We   show that $y  \in R_v$. 
By taking the reduced norm, we have
$$\Nrd_{D /K} ( x )  = y^2(x_0^2-x_1^2a-x_2^2b+x_3^2ab).$$

Since $x$ is integral over $R_v$, 
$\Nrd_{D/K}(x) \in R_v$ and hence $v(\Nrd_{D/K}(x)) \ge  0$. 
Suppose that $y \not\in R_v$. Then $v(y) < 0$ and 
\begin{equation}\label{nrd}
v(x_0^2 -x_1^2a -x_2^2b + x_3^2ab) =  v(\Nrd_{D/K}(x)y^{-2}) \ge  2.
\end{equation} 
%

\noindent
Case 1:   $D$ is unramified at $v$.  
Then $v(a)=v(b)=0$. 
 By going modulo the maximal ideal of $R_v$ and using  \cref{nrd}, 
 we see that $(\overline{x_0}, \overline{x_1}, \overline{x_2}, \overline{x_3}) \in k(v)^4$ is an 
 isotropic vector  for  $\langle \overline{1}, -\overline{a}, -\overline{b},\overline{a}\overline{b}\rangle$. 
Since $K_v$ is a complete discretely valued field,  by a theorem of Springer,
 $\langle 1, -a, -b,ab\rangle$ is isotropic over $K_v$, which contradicts the fact that $D\otimes_KK_v$ is division. 
Hence $ y \in R_v$. 

\noindent
Case 2: $D$ is ramified at $v$. Then $v(a) = 0$ and   $v(b) = 1$.
Since $(\overline{x_0}, \overline{x_1}, \overline{x_2}, \overline{x_3})\ne \vec{0}$, we have $(\overline{x_0}, \overline{x_1})\ne \vec{0}$ or $(\overline{x_2}, \overline{x_3})\ne \vec{0}$ in $k(v)^2$. 

Suppose $(\overline{x_0}, \overline{x_1})\ne \vec{0}$. 
Going modulo the maximal ideal of $R_v$  and using  \cref{nrd}, we see that 
$(\overline{x_0}, \overline{x_1}) \in k(v)^2$ is an isotropic vector for 
 $\langle \overline{1}, -\overline{a}\rangle$. 

Suppose $(\overline{x_0}, \overline{x_1})=\vec{0}$. 
Then $(\overline{x_2}, \overline{x_3})\ne \vec{0}$.
Since $v(x_0) = v(x_1) \ge  1$,
$v(x_0^2 - x_1^2a) \ge  2$. 
Then,  by \cref{nrd}, we have $v(x_2^2b - x_3^2ab) \ge  2$.
Since $v(b) = 1$,  $v(x_2^2 - x_3^2a) \ge  1$. Once again going modulo the maximal ideal 
of $R_v$, we see that $(\overline{x_2}, \overline{x_3}) \in k(v)^2$ is an isotropic vector of $\langle \overline{1}, -\overline{a} \rangle$. 

By a theorem of Springer, $\langle 1, -a \rangle$ is isotropic over $K_v$, which contradicts the fact that $D\otimes_KK_v$ is division. Hence $y \in R_v$. 

In both cases, we have $y \in R_v$. Thus $x \in \Lambda$ and  $\Lambda$ is the unique maximal $R_v$-order in $D$. 
\end{proof}
 
Let $(V, h)$ be an $\varepsilon$-hermitian space over $(D,\sigma)$ for $\varepsilon\in \{1, -1\}$. 
Then there exists an orthogonal basis of $V$ such that $h$ has a diagonal form 
$\langle a_1,\cdots,a_m\rangle$, $a_i\in D$, $\sigma(a_i)=\varepsilon a_i$. 
If $w(a_i)= 0$ for all $i$, then $q_w(h) = \langle q_w(a_1),\cdots,q_w(a_m)\rangle  \in\Herm^{\varepsilon}(D(w), \sigma_{w})$. 
Up to isometry, we may assume that any $h\in\Herm^{\varepsilon}(D,\sigma)$ has diagonal entries with $w$-value either $0$ or $w(t_D)$ \cite[2.20]{L2}. 

\begin{prop}[{\cite[3.4, 3.6]{L1}, \cite[3.27, 3.29]{L2}}]\label{Lamour}
Suppose $\sigma(\pi_D)=\varepsilon'\pi_D$ for $\varepsilon'\in\{1,-1\}$. 
There exists a unique decomposition $h_{K_v}\simeq h_1\perp h_2\pi_D$, where $h_1\in\Herm^{\varepsilon}(D\otimes_KK_v, \sigma\otimes_K\Id_{K_v})$, $h_2\in \Herm^{\varepsilon\varepsilon'}(D\otimes_KK_v, \Int(\pi_D)\circ(\sigma\otimes_K\Id_{K_v}))$ and each diagonal entry of $h_1$ and $h_2$ has $w$-value $0$. 
Furthermore, the following are equivalent: 

(a) $h$ is isotropic;  

(b) $h_1$ or $h_2$ is isotropic; 

(c) $q_w(h_1)$ or $q_w(h_2)$ is isotropic.

\end{prop}

\section{Complete regular local ring of dimension 2}\label{sec7}
 We fix the following notation and assumption throughout this section. 
\begin{itemize}
\item $R$ is a complete regular noetherian local ring of dimension 2, 
\item $K$ is the field of fractions of $R$, 
\item $\mathfrak m=(\pi,\delta)$ is the maximal ideal of $R$, 
\item $k=R/\mathfrak m$, $\Char k\ne 2$, 
\item $L = K(\sqrt{\lambda})$, $\lambda \in R$ with $\lambda = w$, $w\pi$ or $w\delta$ for a unit $w \in R$,
\item $S$ is the integral closure of $R$ in $L$.
\end{itemize}

By the assumption on $\lambda$ and \cite[3.1, 3.2]{PS14}, $S$ is a regular local ring of dimension 2 with maximal ideal $(\pi' , \delta')$, where
\begin{itemize}
\item if $\lambda = w$ is a unit in $R$, then $\pi' = \pi$ and $\delta' = \delta$; 
\item if $\lambda = w\pi$, then $\pi' = \sqrt{w\pi}$ and $\delta' = \delta$; 
\item if $\lambda = w\delta$, then $\pi' = \pi$ and $\delta' = \sqrt{w\delta}$.
\end{itemize}

Let $D$ be a central division algebra over $L$ which is unramified at all height one prime ideals of $S$ except possibly at $\pi'$ and $\delta'$.
 Let $\mathfrak p$ be a height one prime ideal of $S$. 
By \cite[th.~2]{M}, the valuation $v_{\mathfrak p}$ extends to $D$ if and only if $D\otimes_LL_{\mathfrak p}$ is a division algebra. 
Suppose $\deg(D)=d$ and $K$ contains a primitive $d$-th root of unity, by \cite[2.4]{RS}, $D\otimes_LL_{(\pi')}$ and $D \otimes_LL_{(\delta')}$ are division. 
Let $w_{\pi'}$ and $w_{\delta'}$ be the unique extensions of $v_{(\pi')}$ and $v_{(\delta')}$ to $D\otimes_LL_{(\pi')}$ 
and $D \otimes_LL_{(\delta')}$, respectively. 

\subsection{Decomposition of diagonal forms}
\begin{lem}\label{decomp}
Suppose that $\deg(D)=d$, $K$ contains a primitive $d$-th root of unity and $D$ has an involution $\sigma$ (of the first or the second kind) with $L^{\sigma}=K$. 
Suppose there exists a maximal $S$-order $\Lambda$ in $D$ with $\sigma(\Lambda) = \Lambda$ 
and $\pi_D, \delta_D \in \Lambda$ such that

\noindent 
(1) $\Nrd_{D/L}(\pi_D)=u_0\pi'^{d/e_0}$, where $u_0\in R^*$, $e_0=[w_{\pi'}(D^*):v_{\pi'}(L^*)]$ and $e_0$ is invertible in $k$; 
$\Nrd_{D/L}(\delta_D)=u_1\delta'^{d/e_1}$, where $u_1\in R^*$, $e_1=[w_{\delta'}(D^*):v_{\delta'}(L^*)]$ and $e_1$ is invertible in $k$. 

\noindent
(2) $\sigma(\pi_D)= \varepsilon_0 \pi_D$, $\sigma(\delta_D)=\varepsilon_1 \delta_D$ 
 and $\pi_D\delta_D= \varepsilon_2 \delta_D\pi_D$, $\varepsilon_0, \varepsilon_1, \varepsilon_2\in\{1, -1\}$. 

Let $c\in \Lambda$ such that $\sigma(c)=\pm c$ and $\Nrd_{D/L}(c)= u_c\pi'^{dm/e_0}\delta'^{dn/e_1}$ for $u_c \in S^*$, $m, n\in \mathbb Z$. 
 Then 
 $$\langle c\rangle\simeq \langle \theta \pi_D^{m'}\delta_D^{n'}\rangle$$
 for $\theta \in \Lambda^*$ and $m', n'\in\{0,1\}$. 
\end{lem}

\begin{proof}
Since $\Nrd_{D/L}(c)=u_c\pi'^{dm/e_0}\delta'^{dn/e_1}$, it follows that  $w_{\pi'}(c)=mw_{\pi'}(\pi'_D)$ and $w_{\delta'}(c)=nw_{\delta'}(\delta'_D)$. 
Write $m=2r+m'$, $n=2s+n'$ with $m', n'\in \{0,1\}$. 
Let $x=\pi_D^{r}\delta_D^{s}$. 
Then $\sigma(x)=\varepsilon_0^r\varepsilon_1^s( \varepsilon_2)^{rs}x = \varepsilon_cx$, where 
 $\varepsilon_c =\varepsilon_0^r\varepsilon_1^s( \varepsilon_2)^{rs}\in\{1, -1\}$.
By the choice of $\pi_D$ and $\delta_D$, we have $\Nrd_{D/L}(x)=u_0^ru_1^s\pi'^{dr/e_0}\delta'^{ds/e_1}$. 

Let $\theta =\varepsilon_cx^{-1}cx^{-1}(\pi_D^{m'}\delta_D^{n'})^{-1}$. 
Then $c=\sigma(x)(\theta \pi_D^{m'}\delta_D^{n'})x$. 
In particular we have 
 $$\langle c\rangle\simeq \langle \theta \pi_D^{m'}\delta_D^{n'}\rangle.$$
Thus it is enough to show that $\theta \in \Lambda^*$. 

Since $\Lambda=\bigcap_{\mathfrak p}\Lambda_{\mathfrak p}$, where $\mathfrak p$ runs through all height one prime ideals of $S$, we have $\Lambda^*=\bigcap_{\mathfrak p}\Lambda_{\mathfrak p}^*$. 
It suffices to show that $\theta\in \Lambda_{\mathfrak p}^*$ for all height one prime ideals 
$\mathfrak p$ of $S$. 
We have that $\Nrd_{D/L}(\theta) = \Nrd_{D/L}(x)^{-2}\Nrd_{D/L}(c)\Nrd_{D/L}(\pi_D^{m'}\delta_D^{n'})^{-1}
= u_c$ is a unit in $S$. 

\noindent
Case 1: Suppose $\mathfrak p\ne (\pi'), (\delta')$. Since $\pi_D, \delta_D \in \Lambda$ and $\Nrd_{D/L}(\pi_D)$,
$\Nrd_{D/L}(\delta_D)$ are units at ${\mathfrak p}$, by \cite[4.3(c)]{Sal0}, $\pi_D$ and $\delta_D$ are units in 
$\Lambda_{\mathfrak p}$. 
Since $x\in\Lambda_{\mathfrak p}^*$ and $c\in \Lambda$, we have $\theta \in \Lambda_{\mathfrak p}$.
Since $\Nrd_{\Lambda_{\mathfrak p}/S_{\mathfrak p}}(\theta)=\Nrd_{D/L}(\theta)\in S^*$, 
by \cite[4.3(c)]{Sal0}, $\theta \in \Lambda_{\mathfrak p}^*$. 

\noindent
Case 2: Suppose $\mathfrak p=(\pi')$. 
Since $w_{\pi'}(\theta )=0$, by \cref{unique-max}, $\theta \in \Lambda_{(\pi')}^*$. 

\noindent
Case 3: Suppose $\mathfrak p=(\delta')$. 
The proof of $\theta \in \Lambda_{(\delta')}^*$ is similar to Case 2. 
\end{proof}

\begin{cor}\label{decomp2} Let $D$, $\sigma$, $\Lambda$, $\pi_D$ and $\delta_D$ be as in \cref{decomp}.
Let $h$ be an $\varepsilon$-hermitian space over $(D, \sigma)$. 
Suppose that $ h = \langle a_1, \cdots, a_n\rangle$ with $a_i\in \Lambda$ and $\Nrd_{D/L}(a_i)$ is a unit of $S$ times a power of $\pi'$ and a power of $\delta'$ for all $1\le i\le n$. 
Then $$h \simeq \langle u_1, \cdots, u_{n_0}\rangle \perp \langle v_1, \cdots, v_{n_1}\rangle\pi_D \perp \langle w_1, \cdots, w_{n_2}\rangle \delta_D \perp \langle \theta_1, \cdots, \theta_{n_3}\rangle \pi_D\delta_D$$
with $u_i, v_i, \theta_i \in \Lambda^*$ and $n_0+n_1+n_2+n_3=n$.
\end{cor}

\begin{proof} Follows from \cref{decomp}.
\end{proof}

\begin{cor}\label{new}
Under all hypotheses of \cref{decomp2},
if $h\otimes_K1_{K_{\pi}}$ is isotropic over $(D\otimes_KK_{\pi}, \sigma\otimes_K\Id_{K_{\pi}})$ or $h\otimes_K1_{K_{\delta}}$ is isotropic over $(D\otimes_KK_{\delta}, \sigma\otimes_K\Id_{K_{\delta}})$, then $h$ is isotropic over $(D, \sigma)$. 
\end{cor}

\begin{proof}
By \cref{decomp2}, we have 
\[h\simeq h_{00}\perp h_{10}\delta_D\perp h_{01}\pi_D\perp h_{11}\delta_D\pi_D\]
where diagonal entries of $h_{ij}$ are in $\Lambda^*$. 
Applying Larmour's result \cref{Lamour} to $h_{K_{\pi'}}$, we have 
$q_{\pi'}(h_{00}\perp h_{10}\delta_D)$ or $q_{\pi'}(h_{01}\perp h_{11}\delta_D)$ is isotropic over $D(\pi')$. 
Applying \cref{Lamour} again, we obtain that one of $q_{\bar{\delta'}}(q_{\pi'}(h_{ij}))$ is isotropic over $D(\pi')(\bar{\delta'})$. 
Since the diagonal entries of $h_{ij}$ are in $\Lambda^*$, $(D, \Int(\delta_D^{i}\pi_D^{j})\circ\sigma, h_{ij})$ is defined over the maximal $R$-order $\Lambda$ in $D$. 
By \cite[ch.~II, 4.6.1 and 4.6.2]{Knus}, 
one of $h_{ij}$ is isotropic over $(\Lambda, \Int(\delta_D^{i}\pi_D^{j})\circ\sigma|_\Lambda)$. 
Then one of $h_{ij}\delta_D^i\pi_D^j$ is isotropic over $(\Lambda, \sigma|_\Lambda)$. 
Then $h$ is isotropic over $(\Lambda, \sigma|_{\Lambda})$ and hence over $(D, \sigma)$. 
\end{proof}

\begin{cor}\label{new2}
Let $R$, $K$, $S$ and $L$ be as before and let $\iota$ be an automorphism of $L$ such that $\iota|_K=\Id_K$. 
Let $h  = \langle a_1, \cdots, a_n\rangle$ be an $\varepsilon$-hermitian space over $(L, \iota)$ for $\varepsilon \in\{1, -1\}$. 
Suppose that each divisor of $a_i$ is supported only along $\pi'$ and $\delta'$. 
If $h\otimes_K1_{K_{\pi}}$ is isotropic over $(L\otimes_KK_{\pi}, \iota\otimes_K\Id_{K_{\pi}})$ or $h\otimes_K1_{K_{\delta}}$ is isotropic over $(L\otimes_KK_{\delta}, \iota\otimes_K\Id_{K_{\delta}})$, then $h$ is isotropic over $(L/K, \iota)$. 
\end{cor}

\begin{proof}
Let $D = L$, $\sigma=\iota$, $\Lambda = S$, 
$\pi_D = \pi'$ and $\delta_D = \delta'$ in \cref{new}. 
\end{proof}
 
\subsection{Parameters for quaternion algebras}
Suppose $D$ is a quaternion division algebra. The aim of the rest of the section is to show that there exists a maximal order $\Lambda$, $\pi_D$ and $\delta_D$ as in \cref{decomp}. 

We begin with Saltman's classification. 
 
\begin{prop}{\cite[1.2]{Sal1, Sal1.5}, \cite[2.1]{Sal3}}\label{class2}
Suppose $\alpha\in {_2}\Br(K)$. 
If $\alpha$ is unramified at all height one prime ideals of $R$ except possibly at $(\pi)$ and $(\delta)$, then $\alpha$ is of the form $\alpha=\alpha'+\alpha''$, where $\alpha'\in \Br(R)$ and $\alpha$ is described as follows: 

\noindent
(i) If $\alpha$ is unramified at all height one prime ideals of $R$, then $\alpha=\alpha'$; 

\noindent
(ii) If $\alpha$ is ramified only at $(\pi)$, then $\alpha=\alpha'+(u,\pi)$ for some $u\in R^*- {R^*}^2$; 

\noindent
(iii) If $\alpha$ is ramified only at $(\pi)$ and $(\delta)$, then there exists $u,v\in R^*$ such that

(a) $\alpha=\alpha'+(u\pi,v\delta)$; or 


(b) $\alpha=\alpha'+(u,\pi)+(v,\delta)$, where $u$, $v$ and $uv$ are not squares and $u,v$ are in different square classes; or 

(c) $\alpha=\alpha'+(u,\pi\delta)$, where $u$ is not a square. 
\end{prop}

\begin{lem}\label{qua} Let $D$ be a quaternion division algebra over $K$ which is unramified at all height one prime ideals of $R$ except possibly at $(\pi)$ and $(\delta)$.
Then $D$ is isomorphic to one of the following over $K$. 
\begin{enumerate} 
\item $(u, v)$, $u, v \in R^*$; 
\item $(u, v\pi)$, $u\in R^*$ is not a square;
\item $(u, v\delta)$, $u\in R^*$ is not a square;
\item $(u\pi,v\delta)$, $u, v\in R^*$; 
\item $(u, v\pi\delta)$, $u\in R^*$ is not a square and $v\in R^*$. 
\end{enumerate} 
\end{lem}
\begin{proof} 
(1) Suppose $D$ is unramified on $R$. 
By \cite[7.4]{AG}, there exists an Azumaya algebra ${\mathcal D}$ over
$R$ with ${\mathcal D} \otimes_R K \simeq D$. 
Since $D$ is a quaternion algebra over $K$,
${\mathcal D} \otimes_R k$ is a quaternion algebra over $k$. 
Hence ${\mathcal D} \otimes_R k= (a, b)$ for $a, b 
\in k^*$. 
Let $u, v \in R^*$ be lifts of $a, b \in k$. 
Since $R$ is complete, by \cite[6.5]{AG}, $D\simeq(u, v)$. 

(2) Let $\alpha$ be the class of $D$ in ${_2}\Br(K)$.
Suppose that $D$ is ramified on $R$ only at $(\pi)$. 
Then, by \cref{class2}, $ \alpha = \alpha' + (u, \pi)$ for $\alpha' \in \Br(R)$ and $u\in R^*$. 
As in the proof of \cite[2.4]{RS}, we have 
 $\Ind(D) = \Ind(D \otimes K_\pi) = 2( \Ind(\alpha' \otimes K(\sqrt{u})))$. 
 Since $D$ is a quaternion 
 algebra, $\alpha \otimes K(\sqrt{u})$ is split.
Then $\alpha' = (u, v)$ for some $v \in K^*$. 
Since $\alpha'$ is unramified on $R$, we may assume that $v \in R^*$. 
Thus $\alpha = \alpha' \otimes (u, \pi) = (u, v) \otimes (u, \pi) = (u, v\pi)$ in $\Br(K)$. Then $D=(u,v\pi)$. 

(3) Similarly, if $D$ is ramified only at $\delta$, then $D = (u, v\delta)$.

(4) and (5). Suppose that $D$ only ramifies at $\pi$ and $\delta$. 
Then, by \cref{class2}, we have
$\alpha = \alpha' + \alpha''$ with $\alpha'\in\Br(R)$ and $\alpha'' = 
(u\pi, v\delta)$ or $(u, \pi) + (v, \delta)$ or $(u, \pi\delta)$ with $u, v \in R^*$. 

(i) Suppose that $\alpha'' = (u, \pi\delta)$. 
Then as above, it follows that $D = (u, v \pi\delta)$. 

(ii) Suppose that $\alpha'' = (u\pi, v\delta)$. 
Then, as above, we have that $\alpha'' \otimes K(\sqrt{\delta})$ is trivial. Since $\alpha'$ is unramified on $R$, as in the proof of \cite[2.4]{RS}, $\alpha' $ is trivial.
Thus $\alpha = (u\pi, v\delta)$. 

(iii) Suppose that $\alpha'' = (u, \pi) + (v, \delta)$. 
As in the proof of \cite[2.4]{RS}, we have 
$\Ind(\alpha ) = \Ind(\alpha' \otimes K(\sqrt{u}, \sqrt{v})) \cdot [K(\sqrt{u}, \sqrt{v}) : K]$.
Since $\Ind(\alpha) = 2$, we have $[K(\sqrt{u}, \sqrt{v}) : K] \le  2$. 
Since $u$ and $v$ are non-squares in $K$, $u$ and $v$ are in the same square class, a contradiction to \cref{class2}(iii)(b). 
Thus this case does not happen. 
\end{proof}

Next, we consider maximal-orders of certain quaternion algebras. 

\begin{lem}\label{qua-order}
Let $D=(a,b)$ be a quaternion division algebra over $K$ given by $i,j$ such that $i^2=a$, $j^2=b$ and $ij=-ji$. 
Let $\Lambda$ be the $R$-algebra generated by $\{1, i, j, ij\}$. 
If $D$ has one of the forms of \cref{qua}, 
then $\Lambda$ is a maximal $R$-order in $D$. 
\end{lem}
\begin{proof}
By definition, $\Lambda$ is an order in $D$. 
By \cite[1.5]{AG0}, an order of a noetherian integrally closed domain is maximal if and only if it is reflexive and its localization at all height one prime ideals are maximal orders. 
Since $R$ is a regular local ring, it is a noetherian integrally closed domain. 
Since $\Lambda$ is a finitely generated free $R$-module, it is reflexive. 
We show that $\Lambda_{\mathfrak p}$ is a maximal $R_{\mathfrak p}$-order for all height one prime ideals $\mathfrak p$ of $R$. 

\noindent
Case 1: Suppose $\mathfrak p\ne (\pi)$ and $\mathfrak p\ne (\delta)$. 
Then $a, b\in R_{\mathfrak p}^*$ and 
hence $\Lambda_{\mathfrak p}$ is an Azumaya algebra over $R_{\mathfrak p}$. In particular 
$\Lambda_{\mathfrak p}$ is a maximal $R_{\mathfrak p}$-order in $D$. 

\noindent
Case 2: Suppose $\mathfrak p=(\pi)$. 
Then, by \cite[2.4]{RS}, $D\otimes_KK_{\pi}$ is a quaternion division algebra over $K_{\pi}$.  
By \cref{ord}, $\Lambda_{(\pi)}$ is a maximal $R_{(\pi)}$-order in $D$. 

\noindent
Case 3: Suppose $\mathfrak p=(\delta)$. 
Similar to case 2, we can show that $\Lambda_{(\delta)}$ is a maximal $R_{(\delta)}$-order in $D$.
\end{proof}

Next, we construct parameters for certain quaternions with involutions of the first kind. 

\begin{lem}\label{parameters}
Let $D$ be a quaternion division algebra over $K$ having one of the forms of \cref{qua} \textit{except} (5)
and let $\sigma$ be the canonical involution on $D$. 
Let $\Lambda$ be the maximal order as in \cref{qua-order}.

Then there exists $\pi_D, \delta_D\in \Lambda$ such that 

(1) $\Nrd_{D/K}(\pi_D)=u_0\pi^{2/e_0}$ and   $\Nrd_{D/K}(\delta_D)=u_1\delta^{2/e_1}$, where $u_0, u_1\in R^*$,  $e_0=[w_{\pi}(D^*):v_{\pi}(K^*)]$, $e_1=[w_{\pi}(D^*):v_{\delta}(K^*)]$ and $e_0, e_1\in\{1,2\}$; 

(2) $\sigma(\pi_D)=\pm \pi_D$, $\sigma(\delta_D)=\pm \delta_D$, $\sigma(\pi_D\delta_D)=\pm \pi_D\delta_D$ and $\pi_D\delta_D=\pm \delta_D\pi_D$. 
\end{lem}

\begin{proof}
We discuss every case of \cref{qua} except (5). 
In the following, $u, v$ are units and we assume them nonsquare if necessary (to make $D$ a division algebra). We assume that for a quaternion algebra $(a, b)$, $i^2=a$, $j^2=b$, $ij=-ji$. 
If $D=(u,v)$, take $\pi_D=\pi$ and $\delta_D=\delta$; otherwise take $\pi_D$ and $\delta_D$ as follows. 
\begin{center}
\begin{tabular}{|c|c|c|c|c|c|c|c|}
\hline
$D$ & $\pi_D$ & $\delta_D$ & $\Nrd(\pi_D)$ & $\Nrd(\delta_D)$ & $\sigma(\pi_D)$ & $\sigma(\delta_D)$ & $\sigma(\pi_D\delta_D)$ \\
\hline
 $(u,v\pi)$ & $j$ & $\delta$ & $-v\pi$ & $\delta^2$ & $-\pi_D$ & $\delta_D$ & $-\pi_D\delta_D$\\
\hline
$(u, v\delta)$ & $\pi$ & $j$ & $\pi^2$ & $-v\delta$ & $\pi_D$ & $-\delta_D$ & $-\pi_D\delta_D$\\
\hline
$(u\pi, v\delta)$ & $i$ & $j$ & $-u\pi$ & $-v\delta$ & $-\pi_D$ & $-\delta_D$ & $-\pi_D\delta_D$ \\
\hline 
\end{tabular}
\end{center}
Then $\pi_D$ and $\delta_D$ have required properties.
\end{proof}

Next, we construct parameters for certain quaternions with involutions of the second kind.
Suppose that $L/K$ is a degree 2 extension and $D/L$ a quaternion algebra with an involution 
$\sigma$ of second kind. Then,  
by a theorem of Albert (see \cite[2.22]{inv}),  there exists a quaternion algebra $D_0$ over $K$ 
such that $D \simeq D_0 \otimes_K L$ and the involution $\sigma $ maps to the involution 
$\sigma \otimes \iota$ where $\sigma_0$ is the canonical involution of $D_0$ and 
$\iota$ is the non-trivial  automorphism of $L/K$.
 
\begin{lem}\label{parameters2}
Let $L=K(\sqrt{\lambda})$, $S$ and $(\pi', \delta')$ as before. 
Let $D_0$ be a quaternion division algebra over $K$  which is unramified  at 
all height one prime ideals of $R$ except possibly at $(\pi)$ and $(\delta)$. 
If  $D_0 = (u, v\pi\delta)$,  we suppose that $\lambda$ is not a unit in $R$.
Let  $D = D_0 \otimes_K L$.
Let  $\sigma_0$ the canonical involution of $D_0$,
$\iota$ be the non-trivial automorphism of $L/K$ 
and $\sigma = \sigma_0 \otimes_K \iota$. 
If $D$ is division, 
then there exist   a maximal 
$S$-order $\Lambda$ in $D$ which is invariant under 
$\sigma$ and $\pi_D, \delta_D \in \Lambda$ such that 

(1) $\Nrd_{D/L}(\pi_D)=u_0\pi'^{2/e_0}$ and $\Nrd_{D/L}(\delta_D)=u_1\delta'^{2/e_1}$, where $u_0, u_1\in S^*$, $e_0=[w_{\pi'}(D^*): v_{\pi'}(L^*)]$, $e_1=[w_{\delta'}(D^*):v_{\delta'}(L^*)]$ and $e_0, e_1 \in\{1,2\}$; 

(2) $\sigma(\pi_D)=\pm \pi_D$, $\sigma(\delta_D)=\pm \delta_D$, $\sigma(\pi_D\delta_D)=\pm \pi_D\delta_D$ and $\pi_D\delta_D=\pm \delta_D\pi_D$. 
\end{lem}

\begin{proof}
By \cref{qua}, $D_0= (u, v)$, $(u, v\pi)$, $(u, v\delta)$, $(u\pi, v\delta)$ or $(u, v\pi\delta)$ for some $u, v \in R^*$. 
If  $D_0 = (a, b)$, then let $i_0, j_0 \in D_0$  with $i_0^2  = a$, $j_0^2 = b$ and $i_0j_0=-j_0i_0$.
 
There are 3 possible shapes for $\lambda$, i.e. $w$, $w\pi$, $w\delta$ with $w$ a unit.
By the assumption that if $\lambda = w$, then $D_0$ is not of the form $(u, v\pi\delta)$.
Since there are 5 possible shapes of $D_0$, we have $3*5-1=14$ possible combinations.
In each of the cases,  choose  $i$ and $ j$  as in the following two tables. 

\[
\begin{array}{|c||c|c|c|c|c|c|c|c|}
\hline
\lambda & w & w & w & w & w\pi & w\pi & w\delta & w\delta\\
\hline
D_0 & (u,v) & (u,v\pi) & (u,v\delta) & (u\pi,v\delta) & (u,v) & (u,v\delta) & (u,v) & (u,v\pi)\\
\hline
D &  (u,v) & (u,v\pi') & (u,v\delta') & (u\pi',v\delta') & (u,v) & (u,v\delta') & (u,v) & (u,v\pi')\\
\hline
i & \multicolumn{8}{c|}{i_0 \otimes 1}\\
\hline
j & \multicolumn{8}{c|}{j_0 \otimes 1}\\
\hline
\end{array}
\]

\begin{center}
\adjustbox{max width=\textwidth}{
\begin{tabular}{|c||c|c|c|c|c|c|}
\hline
$\lambda$ & $w\pi$ & $w\pi$ & $w\pi$ & $w\delta$ & $w\delta$ & $w\delta$\\
\hline
$D_0$ & $(u,v\pi)$ & $(u\pi,v\delta)$ & $(u,v\pi\delta)$ & $(u,v\delta)$ & $(u\pi,v\delta)$ & $(u,v\pi\delta)$\\
\hline
$D$ & $(u,vw)$ & $(uw,v\delta')$ & $(u,vw\delta')$ & $(u,vw)$ & $(u\pi',vw)$ & $(u,vw\pi')$\\
\hline
$i$ &  $i_0\otimes 1$ & $\frac{1}{\pi}(i_0 \otimes \sqrt{\lambda})$ & $i_0 \otimes 1$  & $i_0 \otimes 1$  & $i_0 \otimes 1$  & $i_0 \otimes 1$  \\
\hline
$j$  & $\frac{1}{\pi}(j_0 \otimes \sqrt{\lambda})$ & $j_0 \otimes 1$ & $\frac{1}{\pi}(j_0 \otimes \sqrt{\lambda})$ 
& $\frac{1}{\delta}(j_0 \otimes \sqrt{\lambda})$
& $\frac{1}{\delta}(j_0 \otimes \sqrt{\lambda})$
& $\frac{1}{\delta}(j_0 \otimes \sqrt{\lambda})$
\\
\hline
\end{tabular}
}

\end{center}
Then it can be checked that  $\pi'$ and $\delta'$ are the only primes in $S$ which might 
divide   $i^2, j^2 \in L$.
Let  $\Lambda=S+Si+Sj+Sij$. Then,  by \cref{qua-order},    $\Lambda$ is a maximal $S$-order  of $D$.
By the choice if $i$ and $j$ we have $\sigma(i) = \pm i$ and $\sigma(j) = \pm j$.
Since $\sigma(S) = S$,   $\sigma(\Lambda) = \Lambda$. 
 
Let  $\pi_D, \delta_D \in \Lambda$  be  as in the proof of \cref{parameters}.  
 Then  $\Lambda$, $\pi_D$ and $\delta_D$ satisfy required properties (1) and (2). 
\end{proof}

\subsection{Local results}
 
\begin{cor}\label{7.47} 
Let $D$ be a quaternion division algebra over $K$ with $\sigma$ the canonical involution and $h$ an $\varepsilon$-hermitian space over $(D, \sigma)$. 
Suppose that $D$ is unramified at all height one prime ideals of $R$ except possibly at $(\pi)$, $(\delta)$ and
$D$ is not of the shape of \cref{qua}(5). 
Let $\Lambda$ be the maximal order as in \cref{parameters}.
Suppose $h$ has a diagonal form $\langle a_1, \cdots, a_n\rangle$ such that $a_i\in \Lambda$ and $\Nrd_{D/K}(a_i)$ is a unit of $R$ times a power of $\pi$ and a power of $\delta$. 
If $h\otimes_K1_{K_{\pi}}$ is isotropic over $(D\otimes_KK_{\pi}, \sigma\otimes_K\Id_{K_{\pi}})$ or $h\otimes_K1_{K_{\delta}}$ is isotropic over $(D\otimes_KK_{\delta}, \sigma\otimes_K\Id_{K_{\delta}})$, then $h$ is isotropic over $(D, \sigma)$. 
\end{cor}

\begin{proof}
Follows from \cref{parameters} and \cref{new}. 
\end{proof}

\begin{cor}\label{7.48}
Let $L = K(\sqrt{\lambda})$, $\lambda = w$, $w\pi$ or $w\delta $ for $w \in R^*$. 
Let $S$ be the integral closure of $R$ in $L$ and the maximal ideal $m' = (\pi', \delta')$ of $S$ as above.
Let $D_0$ be a quaternion division algebra over $K$ having one of the forms of \cref{qua} and $\sigma_0$ the canonical involution on $D_0$. 
When $D_0 = (u, v\pi\delta)$, we suppose that $\lambda$ is not a unit in $R$.
Let $\iota$ be the non-trivial automorphism of $L/K$. 
Let $D=D_0\otimes_KL$ and $\sigma = \sigma_0 \otimes_K \iota$. Suppose that 
$D$ is division.
Let $\Lambda$ be the maximal order as in \cref{parameters2}.
Let $h$ be an $\varepsilon$-hermitian space over $(D, \sigma)$. 
Suppose $h$ has a diagonal form $\langle a_1, \cdots, a_n\rangle$ such that $a_i\in \Lambda$ and $\Nrd_{D/L}(a_i)$ is a unit of $S$ times a power of $\pi'$ and a power of $\delta'$. 
If $h\otimes_K1_{K_{\pi}}$ is isotropic over $K_{\pi}$ or $h\otimes_K1_{K_{\delta}}$ is isotropic over $K_{\delta}$, then $h$ is isotropic over $K$. 
\end{cor}

\begin{proof}
Follows from \cref{parameters2} and \cref{new}. 
\end{proof}

The next corollary is for $\sigma$ of the first kind. 

\begin{cor}\label{local-Z}
Under the hypotheses of \cref{7.47}, let $X$ be a projective homogeneous space under $G=\SU(D,\sigma, h)$ 
over $K$. 
If $X(K_{\pi})\ne\emptyset$ or $X(K_{\delta})\ne\emptyset$, then $X(K)\ne \emptyset$.
\end{cor}

\begin{proof}
First we assume that $X$ is of the shape of the 2nd, 3rd, 4th or first part of the 5th case of \cref{phs}.

By \cite[2.4]{RS}, $\Ind(D)=\Ind(D\otimes_KK_{\pi})=\Ind(D\otimes_KK_{\delta})$. 
Then $\Ind(D\otimes_KK_{\pi})|g$ iff $\Ind(D\otimes_KK_{\delta})|g$
iff $\Ind(D)|g$, where $g=\gcd\{n_1,\cdots,n_r\}$.
Let $t=n_r$. 
By \cref{2.1}, it suffices to show that if $X_t(K_{\pi})\ne\emptyset$ or $X_t(K_{\delta})\ne\emptyset$, then $X_t(K)\ne \emptyset$. 
Suppose $X_t(K_{\pi})\ne\emptyset$. 
Then $h_{K_{\pi}}$ has a totally isotropic subspace of reduced dimension $t$, where $t$ is even. 
Then $h_{K_{\pi}}$ is isotropic over $D$. 
So, by \cref{7.47}, $h:V\times V\to D$ is isotropic over $D$. 
Let $x\in V$, $x\ne 0$ be an isotropic vector of $h$. 
Let $i_W$ denote the Witt index. 
Then $2\le\rdim_D(xD)\le t\le 2i_W(h_{K_{\pi}})$, where $\rdim_D(xD)$ is even. 

We induct on $t$. 
If $t=2$, then $\rdim_D(xD)=2$, we have $xD\in X_t(K)$ and hence $X_t(K)\ne \emptyset$. 
%

Now we suppose $t>2$. 
If $\rdim_D(xD)=t$, then $xD\in X_t(K)$ and hence $X_t(K)\ne \emptyset$. 
If $\rdim_D(xD)<t$, by \cite[ch.1, 3.7.4]{Knus}, there exists a hyperbolic plane $\mathbb H\subset (V,h)$ such that $x\in \mathbb H$ and $h=h'\perp\mathbb H$. 
Then by \cite[p.73]{inv}, $2i_W(h'_{K_{\pi}})=2i_W(h_{K_{\pi}})-2\ge 2i_W(h_{K_{\pi}})-\rdim_D(xD)\ge t-\rdim_D(xD)>0$. 
Write $X'_t$ for the corresponding projective homogeneous variety under $\SU(D, \sigma, h')$ over $K$. 
Then $X'_{t-\rdim_D(xD)}(K_{\pi})\ne\emptyset$. 
Since $t-\rdim_D(xD)<t$, by induction, we have $X'_{t-\rdim_D(xD)}(K)\ne\emptyset$. 
Suppose $N\in X'_{t-\rdim_D(xD)}(K)$. 
Then $N\oplus xD\in X_{t}(K)$. 
Hence $X_{t}(K)\ne\emptyset$. 

Therefore $X_t(K_{\pi})\ne \emptyset$ implies $X_t(K)\ne \emptyset$. 
Similarly, $X_t(K_{\delta})\ne \emptyset$ implies $X_t(K)\ne \emptyset$. 

Next we assume that $X$ is of the shape of the second part of the 5th case of \cref{phs}, now $t=n=n_r$. 
We need to prove the following

Subcase ($+$): If $X_n^+(K_{\pi})\ne\emptyset$ or $X_n^+(K_{\delta})\ne\emptyset$, then $X_n^+(K)\ne \emptyset$;  

Subcase ($-$): If $X_n^-(K_{\pi})\ne\emptyset$ or $X_n^-(K_{\delta})\ne\emptyset$, then $X_n^-(K)\ne \emptyset$. 

Suppose $X_n^+(K_{\pi})\ne\emptyset$. 
Then $h_{K_{\pi}}$ is hyperbolic. 
By \cref{7.47} with Witt decomposition, Witt cancellation and induction, $h$ is hyperbolic. 
Then $X_n(K)=X_n^+(K)\sqcup X_n^-(K)\ne\emptyset$. 
If $X_n^+(K)\ne\emptyset$ we are done. 
If $X_n^-(K)\ne\emptyset$, then $X_n^-(K_{\pi})\ne\emptyset$. 
Then both $X_n^+(K_{\pi})\ne\emptyset$ and $X_n^-(K_{\pi})\ne\emptyset$. 
By \cref{pmkey}, we have $D_{K_{\pi}}$ is split. 
By \cite[2.4]{RS}, $D$ is split over $K$, a contradiction to our assumption that $D$ is division. 
Hence, $X_n^+(K)\ne\emptyset$ and $X_n^-(K)=\emptyset$. 

The proof for the subcase ($-$) is similar. 
\end{proof}

The next corollary is for $\sigma$ of the second kind. 

\begin{cor}\label{local-Z2}
Under the hypotheses of \cref{7.48}, let $X$ be a projective homogeneous space under $G=\U(D,\sigma, h)$ over $K$ (see the first case of \cref{phs}). 
If $X(K_{\pi})\ne\emptyset$ or $X(K_{\delta})\ne\emptyset$, then $X(K)\ne \emptyset$.
\end{cor}

\begin{proof}
The proof is similar to the first half of \cref{local-Z} (for the 2nd, 3rd, 4th and the first part of the 5th cases of \cref{phs}), using \cref{7.48}.
\end{proof}

\begin{cor}\label{local-Z3}
	Under the hypotheses of \cref{new2}, let $X$ be a projective homogeneous space under $G=\U(L,\iota, h)$ over $K$. 
	If $X(K_{\pi})\ne\emptyset$ or $X(K_{\delta})\ne\emptyset$, then $X(K)\ne \emptyset$.
\end{cor}

\begin{proof}
	The proof is similar to the first half of \cref{local-Z}, using \cref{new2}.
\end{proof}
 
\section{The main theorem}\label{sec9}
In this section, we prove \cref{main-thm}. 
The next lemma deals with the last case of \cref{qua} to make it possible to apply \cref{parameters} in the proof of \cref{9.3}. 

\begin{lem}\label{9.1}
Let $R$ be a regular local ring with field of fractions $K$, maximal ideal $(\pi, \delta)$ and residue field $k$ with $\Char k\ne 2$. 
Suppose $\alpha = (u, v\pi\delta) \in {_2}\Br(K)$.
Let $\mathscr X = \Proj (R[x, y]/(\pi x - \delta y))\to \Spec(R)$ be the blow-up of $\Spec(R)$ at its maximal ideal. 
For every closed point $Q$ of $\mathscr X$, let $\mathfrak m_Q$ be the maximal ideal of $\mathscr O_{\mathscr X, Q}$. 
Then $\alpha = (u, t)$ for $t\in \mathscr O_{\mathscr X, Q}$ such that $t$ is either a unit or a 
regular parameter (i.e. $t\in \mathfrak m_Q- \mathfrak m_Q^2$).
\end{lem}

\begin{proof} 
Let $Q_1$ be the closed point given by the homogeneous ideal 
$(\pi, x,y)$ and $Q_2$ the closed point given by the homogeneous ideal $(\delta, x,y)$. 
Let $t = \frac{x}{y} \in K$. 
Then $\delta = t\pi$ in $K$. 
Hence at $Q_1$, $t$ is a regular parameter and $\alpha = (u, v\pi\delta) = (u, vt\pi^2) = (u, t)$.
Similarly, at $Q_2$, $1/t$ is a regular parameter and $\alpha = (u, 1/t)$.
Let $Q$ be a closed point of $\mathscr X$ that is neither $Q_1$ nor $Q_2$. 
Then at $Q$, $t$ is a unit and $\alpha = (u, t)$. 
\end{proof}

The next lemma deals with $\lambda$ from \cref{parameters2} to make it possible to apply \cref{parameters} in the proof of \cref{9.3}. 

\begin{lem}\label{9.2}
Let $R$ be a regular local ring  of dimension 2 with field of fractions $K$ 
and residue field $k$ with $\Char k \ne 2$. 
Let $\lambda \in K$ and $\alpha \in {_2}\Br(K)$. 
Then there exists a sequence of blow-ups $\mathscr X \to \Spec(R)$ such that for every closed point $P$ of $\mathscr X$,
the maximal ideal $\mathfrak m_P$ of $\mathscr O_{\mathscr X, P}$ is given by $\mathfrak m_P=(\pi, \delta)$, $\lambda = w$, $w\pi$ or $w\delta$, up to squares for $u\in \mathscr O_{\mathscr X,P}^*$ and $\alpha = \alpha' + \alpha''$ with $\alpha'$ and $\alpha''$ as in \cref{class2}. 
Furthermore, if $\alpha'' = (u, v\pi\delta)$ for units $u, v\in R^*$, then $\lambda\not\in \mathscr O_{\mathscr X,P}^*$, up to squares. 
\end{lem} 
 
\begin{proof} 
By choosing a finite sequence of blow-ups $\mathscr X \to \Spec (R)$,
we may assume that for every closed point $P$ of $\mathscr X$, $\mathfrak m_P=(\pi, \delta)$, $\lambda = w$, $w\pi$, $w\delta$ or $w\pi\delta$, up to squares, for $w\in\mathscr O_{\mathscr X,P}^*$ and $\alpha$ is unramified at $P$ except possibly at $\pi$ and $\delta$. 
In fact, let $P$ be a closed point of $\mathscr X$ such that $\mathfrak m_P=(\pi, \delta)$ and $\lambda = w \pi\delta$ for some unit $w$ of $\mathscr O_{\mathscr X, P}$. 
Let $\mathscr X'$ be the blowup of $\mathscr X$ at $P$ and $Q$ a closed point on the exceptional curve. 
By \cref{9.1}, $\lambda = wt$ or $w'$, up to squares, for units $w$ and $w'$ and $t$ is either a unit or regular parameter.
Since there are only finitely many closed points on $\mathscr X$ with $\lambda = w\pi\delta$, 
we have a finite sequence of blowups $\mathscr X'\to \mathscr X$ such that for every closed point $P'$ of $\mathscr X'$, $\mathfrak m_{P'}$, $\lambda$ and $\alpha$ has the desired property at $P'$. 
In particular, $\alpha = \alpha' + \alpha''$ with $\alpha'$ and $\alpha''$ as in \cref{class2}.

Suppose there exists a closed point $P$ of $\mathscr X'$ such that 
$\alpha'' = (u, v\pi\delta)$ and $\lambda = w$, for $u, v, w \in \mathscr O_{\mathscr X',P}^*$.
Let $\mathscr X'' \to \mathscr X'$ be the blow-up at $P$ as in \cref{9.1}. 
Then for every closed point $Q$ of the exceptional curve of $\mathscr X''$, by \cref{9.1}, we have $\alpha'' = (u, v)$ or $(u, t)$ for a regular parameter $t$ at $Q$ and $u, v\in\mathscr O_{\mathscr X'',Q}^*$. 
Since $\lambda = w\in\mathscr O_{\mathscr X',P}^*$, it remains a
unit in $\mathscr O_{\mathscr X'',Q}^*$. 
Since there are only finitely many closed points with $\alpha'' =(u, v\pi\delta)$, we have the required sequence of blow-ups $\mathscr X''\to \Spec(R)$. 
\end{proof}

\begin{lem}  
\label{blowing_up}
Let $R$ be a regular local ring of dimension 2 with field of fractions $K$ 
and residue field $k$. Suppose $\Char k \ne 2$. 
Let $\widehat{R}$ be the completion of $R$ at its maximal ideal and
$\widehat{K}$ the field of fractions of $\widehat{R}$. 
Let $\mu \in \widehat{K}^*$.
Then there is a finite sequence of blow-ups $\mathscr X \to  \Spec(R)$ such that 
for every closed point $Q$ of  $\mathscr X \times_{\Spec{R}}\Spec(\widehat{R})$,
the maximal ideal at $Q$ is given by $(\pi, \delta)$ with the support of $\mu$ at $Q$
is at most $(\pi)$ and $(\delta)$. 
Also, either $(\pi)$ or $(\delta)$ corresponds to an exceptional curve 
in $\mathscr X$. 
\end{lem}

\begin{proof} Since $\widehat{R}$ is a regular local ring of dimension 2, 
there exists a finite sequence of blow-ups  $\widehat{\mathscr X} \to \Spec{\widehat{R}}$  at the closed point of 
 $\Spec(\widehat{R})$ and closed points on the exceptional curves such that the 
support of $\mu$ on $\widehat{\mathscr X}$ is a union of regular curves with normal crossings \cite{Abh} or \cite{Lip75}. 
 Since any exceptional curve is the projective line over a finite extension of $k$, there exists a finite sequence of blow-ups $\mathscr X  \to \Spec(R)$
 such that $\mathscr X \times_{\Spec(R)} \Spec{\widehat{R}} = \widehat{\mathscr X}$
 (see \cite[prop.~3.6]{HHK5}).
 
 Let $Q$ be a closed point  of $\widehat{\mathscr X}$.
 Then, by the choice of $\widehat{\mathscr X}$, 
 the maximal ideal at $Q$ is given by $(\pi,\delta)$ and the support of $\mu$ at $Q$ is 
 at most $(\pi)$ and $(\delta)$. Suppose that neither $(\pi)$ nor $(\delta)$ is an 
 exceptional curve. Then blow-up $Q$. The resulting sequence of 
 blow-ups has required properties. 
 \end{proof}

\begin{thm}\label{9.3}
Let $K$ be a complete discrete valued field with residue field $k$, $\Char k\ne 2$. 
Let $F$ be the function field of a smooth, projective, geometrically integral curve over $K$.
Let $L/F$ be an extension of degree at most 2 and 
$A$  a finite-dimensional simple $F$-algebra with center $L$. 
Let $\sigma$ be an involution on $A$ such that $F=L^{\sigma}$.  
Let $h: V\times V\to A$ be an $\varepsilon$-hermitian space over $(A, \sigma)$ for $\varepsilon\in\{1, -1\}$. 
Let 
\[G(A, \sigma, h)=\left\{
\begin{array}{ll}
\SU(A,\sigma, h) & \text{if } \sigma \text{ is of the first kind;} \\
\U(A, \sigma, h) & \text{if } \sigma \text{ is of the second kind.}\\
\end{array}
\right.\]
 Suppose that for any regular proper model $\mathscr X$ of 
$F$ and for any closed point $P$ of $\mathscr X$
 $\Ind(A \otimes F_P)\le 2$. 
Then the Hasse principle holds for any projective homogeneous space under $G(A, \sigma, h)$.
\end{thm}
 
%

\begin{proof} 
 
Let $X$ be a projective homogeneous space under $G(A, \sigma, h)$. 
Suppose that $X(F_v) \ne \emptyset$ for all divisorial discrete valuations of $F$.
We use \cite[3.7]{HHK1} to show that $X(F) \ne \emptyset$. 
Since $\sigma$ is arbitrary, we assume that $\varepsilon = 1$.
We use patching notations from \cref{data}. 

Write $L = F(\sqrt{\lambda})$ for $\lambda \in F^*$. 
Let  $\mathscr X$ be  a regular proper model of $F$ such that the union of the support of $\lambda$ and
the special fiber $\mathscr X_1$ of $\mathscr X$ is a union of regular curves with normal crossings. 
Let $\eta$ be a codimension zero point of $\mathscr X_1$.  
Since $X(F_\eta) \ne \emptyset$, by \cite[5.8]{HHK2}, there exists a non-empty open subset 
$U_\eta$ of the closure of $\eta$ such that $X(F_{U_\eta}) \ne \emptyset$ and $U_\eta$ 
does not meet other regular curves in the special fiber $\mathscr X_1$.

Let $\mathcal P$ be the finite set of closed points of $\mathscr X_1$ which are not on $U_\eta$
for any codimension zero point $\eta$ of $\mathscr X_1$.
For $P \in \mathcal P$, let $D_P$ be the central division algebra over $L_P = L \otimes F_P$
which is Brauer equivalent to $A \otimes F_P$. 
By Morita equivalence \cite[ch.~I, 9.3.5]{Knus}, there exists an involution $\sigma_P$ on $D_P$ and $h$ corresponds to a hermitian form $h_P$ over $(D_P, \sigma)$. 

Since for any closed point $P$ of $\mathscr X$, $\deg(D_P) \le  2$, either $D_P= L_P$ or $D_P$ is 
a quaternion division algebra. 
If $[L : F] = 2$, since $L^{\sigma}=F$,   $L_P^{\sigma_P}=F_P$
and   by a theorem of Albert \cite[2.22]{inv}, there exists a central division  algebra $(D_P)_0$ over $F_P$ 
such that $\deg ((D_P)_0) \le 2$ and $D_P \simeq (D_P)_0 \otimes L_P$. If $\deg((D_P)_0) = 2$,  then write 
$(D_P)_0 = (a_P, b_P)$ for some $a_P, b_P \in F_P$.

By \cref{blowing_up}, there exists a finite sequence of blow-ups $ \phi  : \mathscr X' \to \mathscr X$
such that for each $P \in \mathcal P$ and  $Q \in \phi^{-1}(P)$, the support  of 
$a_P$ and $b_P$ at $Q$ have normal crossings. In particular the ramification divisor of
$(D_P)_0$ at $Q$ has normal crossings.  
Let $\eta$ be an exceptional curve in $\mathscr X'$.  Since $X(F_\eta) \ne \emptyset$,
as above there exists a non-empty open set $U_\eta$ of the closure of $\eta$
such that $X(F_{U_\eta}) \ne \emptyset$.  Let $Q \in \mathscr X'$ be in the closure of 
$\eta$. Suppose $D \otimes F_Q$ is non-split. 
Since  $\phi(Q) = P$ and $D \otimes F_Q$ is Brauer equivalent to $D_P \otimes F_Q = (D_P)_0 \otimes L\otimes F_Q$.
In particular the support of the ramification divisor of $(D_P)_0 \otimes F_Q$ has normal crossings.
Thus, replacing $\mathscr X$ by $\mathscr X'$, we assume that if $P  \in \mathcal{P}$,
then $D_P = (D_P)_0 \otimes L_P$ and  the ramification divisor of $(D_P)_0$ has normal crossings at $P$.
Further, replacing $\mathscr X$ by a finite sequence of blow-ups at the points of $\mathcal P$,
using \cref{9.2}, we assume that for $P \in \mathcal P$, $D_P$ and $\lambda$ are
as in \cref{9.2}.

Let $P \in \mathcal{P}$. 
If $D_P = L_P$, let $\Lambda_P $ be the integral closure of $\widehat{R_P}$ in $L_P$.
If $D_P \ne L_P$, then $D_P = (D_P)_0 \otimes L_P$ with $(D_P)_0$ a quaternion algebra 
and $(D_P)_0$, $\lambda$ are as in \cref{9.2}. 
Let $\Lambda_P$ be the order as in \cref{parameters} or \cref{parameters2}. 
Since $D_P$ is division, $h_P = \langle a_1^P, \cdots , a_m^P\rangle$ with $a_i^P \in \Lambda_P$  
and $\sigma_P(a_i^P)  = a_i^P$. Let $f_i^P = \Nrd_{D_P}(a_i^P) \in F_P \subset L_P$. 
Since $\sigma_P(a_i^P) = a_i^P$, $f_i^P \in F_P$. 
Once again, using \cref{blowing_up}, replacing $\mathscr X$
by a finite sequence of blow-ups of $\mathscr X$ at the points of $\mathcal P$, we assume
that for every $P \in \mathcal P$,   the maximal ideal at $P$ is given by 
$(\pi_P, \delta_P)$, the support of $f_i^P$ is at most $\pi_P$ and $\delta_P$ and
at least one of $\pi_P$ and $\pi_P$ is an exceptional curve. 

Let $X^P$ be the projective homogeneous space under $G(D_P, \sigma_P, h_P)$. 
The maximal ideal at $P$ is given by $(\pi_P, \delta_P)$ and either $\pi_P$ or
$\delta_P$, say $\pi_P$, gives an exceptional curve. Since the valuation given by an exceptional curve is 
a divisorial discrete valuation, $X(F_{\pi_P}) \ne \emptyset$.
Thus, by \cref{2.2} or \cref{2.4},  $X^P((F_P)_{\pi_P})\ne \emptyset.$
 If $D_P = L_P$, then, by \cite[3.1]{CTPS} or \cref{local-Z3}, $X(F_P) \ne \emptyset$. 
If $D_P$ is a quaternion algebra, then, by \cref{local-Z} or \cref{local-Z2}, $X^P(F_P) \ne \emptyset$. 
By \cref{2.2} or \cref{2.4} again, $X(F_P) \ne \emptyset$ for all $P\in\mathcal P$. 

Therefore, by \cite[3.7]{HHK1}, $X(F)\ne \emptyset$.
\end{proof}
 
Now we restate and prove \cref{main-thm} as \cref{thm1*}. 

\begin{thm}
\label{thm1*}
Let $K$ be a complete discrete valued field with residue field $k$, $\Char k\ne 2$. 
Let $F$ be the function field of a smooth, projective, geometrically integral curve over $K$.
Let $\Omega$ be the set of all rank one discrete valuations on $F$. 
For each $v\in \Omega$, let $F_v$ be the completion of $F$ at $v$. 
Let $A$ be a finite-dimensional simple $F$-algebra with an involution $\sigma$ such that $F=Z(A)^{\sigma}$. 
Suppose that at least one of the following is satisfied. 
 
(1) $\Ind(A)\le 2$;

(2) $\Per(A)=2$, $|l^*/{l^*}^2|\le 2$ and ${_2}\Br(l)=0$ for all finite extensions $l/k$.

Let $\varepsilon\in\{1, -1\}$ and $h: V\times V\to A$ an $\varepsilon$-hermitian space over $(A, \sigma)$. 
Let $X$ be a projective homogeneous space under  
\[G=\left\{
\begin{array}{ll}
\SU(A,\sigma, h) & \text{if } \sigma \text{ is of the first kind;} \\
\U(A, \sigma, h) & \text{if } \sigma \text{ is of the second kind.}\\
\end{array}
\right.\]
If $X(F_v)\ne\emptyset$ for all $v\in \Omega$, then $X(F) \ne \emptyset$. 
\end{thm}

\begin{proof}

%
%
Let $L=Z(A)$. Let $\mathscr X$ be a regular proper model of $L$ 
with ramification locus of $A$ a union of regular curves with 
normal crossings and $P$ a closed point of $\mathscr X$. 
Let $\widehat{R_P}$ and $F_P$ be as in \cref{data} and $L_P=L\otimes F_P$. 
Let $k_P$ be the residue field of $\widehat{R_P}$. 

(1) If $\Ind(A)\le 2$, we have $\Ind (A\otimes L_P)\le 2$ for all closed points $P$ of $\mathscr X$. 

(2) Suppose $\Per(A)=2$, $|l^*/{l^*}^2|\le 2$ and ${_2}\Br(l)=0$ for all finite extensions $l/k$. 
Then $k_P^*$ has at most two square classes and ${_2}\Br(k_P)=0$. 
Then by \cite[6.2]{AG}, ${_2}\Br(\widehat{R_P})=0$. 
Then, by \cref{class2}, $\Ind(A\otimes L_P)\le 2$. 

Hence the Hasse principle is a consequence of \cref{9.3}. 
\end{proof}

Finally, we restate and prove \cref{main-cor} as \cref{cor1*}. 

\begin{cor}\label{cor1*}
Let $p$ be an odd prime. 
Let $K$ be a $p$-adic field. 
Let $F$ a function field in one variable over $K$. 
Let $\Omega$ be the set of all discrete valuations on $F$. 
Let $G$ be a connected linear algebraic group such that there exists an isogeny from a product of almost simple groups of one of the following types to the semisimple group $G/\Rad(G)$. $${^1}A_n, \quad {^2}A_n^*,\quad  B_n, \quad C_n, \quad D_n~(D_4 \text{ nontrialitarian}),$$ where ${^2}A_n^*$ means that the almost simple factor is isogenous to a unitary group $\U(A, \sigma, h)$ such that $\sigma$ is of the second kind and $\Per(A)=2$. 
Let $X$ be a projective homogeneous space under $G$. 
Then \[\prod\limits_{v\in\Omega} X(F_v)\ne\emptyset\implies X(F)\ne\emptyset. \]
\end{cor}

\begin{proof}
Let $G^{ss}$ be the semisimple group $G/\Rad(G)$. 
By \cite[5.7]{CTGP}, $X$ is a projective homogeneous space under $G^{ss}$.  
By \cite[14.10(2)]{Borel}, there exists an isogeny $G_1\times \cdots \times G_r\to G^{ss}$ where $G_i$ are almost simple groups. 
Since $\Char F=0$, all isogenies of algebraic groups over $F$ are central. 
By \cite[2.20, (i)]{BT}, central sujective morphisms of algebraic groups give isomorphisms of their projective homogeneous spaces. 
Then $X$ is a projective homogeneous space under $G_1\times \cdots \times G_r$. 
By \cite[6.10(e)]{MPW2},  $X\simeq X_1\times \cdots \times X_r$ where $X_i$ is a projective homogeneous space under $G_i$ for each $1\le i\le r$. 
Then $X(F)\ne\emptyset$ if and only if $X_i(F)\ne \emptyset$ for all $1\le i\le r$. 
By assumption, $G_i$ has one of the types ${^1}A_n$, ${^2}A_n^*$, $B_n$, $C_n$, ${^1}D_n$, ${^2}D_n$. 
The type ${^1}A_n$ case has been proved by Reddy and Suresh \cite[2.6]{RS}. 
The type $B_n$ case has been proved by Colliot-Th\'el\`ene, Parimala and Suresh \cite[3.1]{CTPS}. 
By \cite[Table 1]{Tits}, if $G_i$ has type ${^2}A_n^*$, then $G_i$ is isogenous to $\U(A,\sigma, h)$; if $G_i$ has type $B_n$, $C_n$ or $D_n$, then $G_i$ is isogenous to $\SU(A,\sigma, h)$. 
By \cite[2.20, (i)]{BT} again, we may assume that $G_i$ is the unitary group or the special unitary group as above and hence $X$ is as in  \cref{phs}. 
The rest follow from \cref{thm1*}. 
\end{proof}

\section{Application: Odd degree extensions}\label{sec10}
%
%
In this section we prove \cref{odd}. 
We begin with the following. 

\begin{lem}\label{pre-odd}
Let $(L, v)$ be a complete discrete valued field and 
 $k_L$   the residue field  of $L$ with $\Char k_L\ne 2$. 
Let $M$ be an odd degree extension of $L$, with residue field $k_M$. 
Suppose that for any central division $k_L$-algebra $E$ with involution $\tau$ and any $\varepsilon$-hermitian form $\varphi$ over $(E, \tau)$, if $\varphi_{k_M}$ is isotropic, then $\varphi$ is isotropic. 
Let $D$ be a central division $L$-algebra with an involution $\sigma$ and $\Per(D)=2$. 
Let $h$ be an $\varepsilon$-hermitian form over $(D, \sigma)$ for $\varepsilon\in\{1, -1\}$. 
If $h_M$ is isotropic, then $h$ is isotropic.
\end{lem}

\begin{proof}
Since $L$ is complete, the valuation $v$ on $L$ extends to a discrete valuation $v'$ on $M$. 
Let $t$ be a uniformizer of $L$, $t'$ a uniformizer of $M$ such that $(t')^e=t$ where $e=e(M/L)$. 
By \cite[4.5.11, 2.]{GS}, $D'=D\otimes_LM$ is a division algebra. 
Let $w$ be the extension of $v$ to $D$ and $w'$ the extension of $v'$ to $D'$. 
Let $\pi$ be a uniformizer of $D$ and $\pi'$ a uniformizer of $D'$. 
By \cite[prop.~2.7]{L2}, there exists $x\in D$   such that 
\begin{equation}\label{unif1}
w(x)\equiv w(\pi)\mod 2w(D^*),\quad\sigma(x)=\varepsilon x, \varepsilon\in\{1,-1\}.
\end{equation}
By the second to the last paragraph of \cite[p.~393]{W3}, $e(D'/D)$ is a factor of $[M:L]$.
Since $[M:L]$ is odd,  $e(D'/D)$ is odd. 
Then $w'(\pi\otimes_L1_M)\equiv w'(\pi')\mod 2w'({D'}^*)$. 
Let $x' = x \otimes 1 \in D'$ and  $\sigma(x') = \varepsilon x'$.
By Larmour's theorem, \cref{Lamour}, 
\begin{equation}\label{dec1}
h\simeq h_1\perp h_2x
\end{equation}
where all diagonal entries of $h_1$ and $h_2$ have valuation $0$ in $D$. 
Thus 
\begin{equation}\label{dec2}
h_M\simeq(h_1)_M\perp (h_2)_M(x\otimes_L1_M)=(h_1)_M\perp (h_2)_M x'
\end{equation}
In the following, an overline means ``over the residue field''. 
We have 
\begin{center}
\begin{tabular}{ll}
& $h_M$ is isotropic, \\
$\iff$ & one of $\overline{(h_i)_M}$ is isotropic over 
$(\overline{ D\otimes_L M}, \overline{\sigma\otimes_L\Id_M})$, \\
& by applying \cref{Lamour} to \cref{dec2}. \\
$\iff$ & one of $(\overline{h_i})_{k_M}$ is isotropic over $(\overline D\otimes_{k_L} {k_M}, \overline{\sigma}\otimes_{k_L}\Id_{k_M})$. \\
$\iff$ & one of $\overline{h_i}$ is isotropic over $(\overline D, \overline{\sigma})$, by the given condition on $k_M/k_L$. \\
$\iff$ & $h$ is isotropic over $(D, \sigma)$, by applying \cref{Lamour} to \cref{dec1}. \\
\end{tabular}
\end{center}where $i\in \{1,2\}$.
\end{proof}

\begin{ww}\label{ready}
Let $L$ be an arbitrary field of characteristic not 2. 
Let $M$ be an odd degree extension of $L$. 
For each discrete valuation $v$ of $L$ with valuation ring $R_v$ and maximal ideal $\mathfrak p_v$, let $\widehat{R_v}$ be its completion and $L_v=\Frac(\widehat{R_v})$. 
Let $S$ be the integral closure of $R_v$ in $M$ and $\mathfrak P_i, 1\le i\le n$ be prime ideals of $S$ lying over $\mathfrak p_v$. 
Let $\widehat{S_i}$ be the completion of $S$ at $\mathfrak P_i$ and $M_i=\Frac(\widehat{S_i})$. 
By \cite[p.~15, (2)]{CF}, $$M\otimes_LL_v\simeq \prod\limits_{i=1}^n M_i.$$
Since $[M:L]=[M\otimes_LL_v:L_v]=\sum\limits_{i=1}^n[M_i: L_v]$ is odd, there exists $1\le j\le n$ such that $[M_j: L_v]$ is odd. 
\end{ww}

\begin{lem}\label{local2}
Let $L$ be a non-archimedean \textit{local} field of characteristic not 2. 
Let $M$ be an odd degree extension of $L$. 
Let $D$ be a division algebra over $L$ such that $D\ne L$. 
Let $\sigma$ be an involution of $D$. 
Let $h$ be an $\varepsilon$-hermitian form over $(D, \sigma)$. 
If $h_M$ is isotropic, then $h$ is isotropic.
\end{lem}

\begin{proof}
Let $\sigma$ be of the first kind. 
By \cite[ch.~10, 2.2(i)]{Sch}, $D$ is the unique quaternion division algebra over $L$, and it suffices to apply \cite[3.5]{PSS}. 

Let $\sigma$ be of the second kind. 
If $\varepsilon=-1$, by Hilbert 90 \cite[ch.~V, \S~11, no.~6, th.~3, a)]{Balg47}, there exists $\mu\in Z(D)- L$ such that $\sigma(\mu)=-\mu$. 
By scaling \cite[ch.~I, 5.8]{Knus}, $h$ is isotropic over $(D, \sigma)$ if and only if $\mu^{-1}h$ is isotropic over $(D, \sigma)$, where $\Int(\mu)\circ \sigma=\sigma$ and $\mu^{-1}h$ is a hermitian form. 
Hence we may assume that $\varepsilon=1$. 
By \cite[ch.~10, 2.2(ii)]{Sch}, $D/L$ is a quadratic field extension. Also $D_M/M$ is a quadratic field extension. 
The conclusion follows from applying Springer's theorem \cite{Spr1} to the trace quadratic forms of $h$ and $h_M$ \cite[ch.~10, 1.1(i)]{Sch}, respectively. 
\end{proof}

\begin{lem}\label{global2}
Let $L$ be a \textit{global} field of characteristic not 2. 
Let $M$ be an odd degree extension of $L$. 
Let $D$ be a division $L$-algebra with an involution $\sigma$ such that $D\ne L$ and $\Per(D)=2$. 
Let $h$ be an $\varepsilon$-hermitian form over $(D, \sigma)$. 
If $h_M$ is isotropic, then $h$ is isotropic.
\end{lem}

\begin{proof}
If $\sigma$ is of the first kind, by \cite[ch.~10, 2.3(vi)]{Sch}, $D$ is a quaternion division algebra and the result follows from \cite[3.5]{PSS}. 

Now suppose $\sigma$ is of the second kind. 
Suppose $Z(D)=L(\sqrt{\lambda})$. 
Let $\Omega_L$ be all the places of $L$ and $\Omega_M$ all the places of $M$. 
If $v\in\Omega_L$ such that $\lambda$ is a square in $L_v$, by \cite[ch.~10, 6.3]{Sch} $(D\otimes_L L_v, \sigma\otimes_L\Id_{L_v})$ is a hyperbolic ring and $h_{L_v}$ is always hyperbolic. 

Suppose $v\in\Omega_L$ is such that $\lambda$ is not a square in $L_v$. 
by \cref{ready} we have an odd degree extension $M_j/L_v$. 

\noindent
Case 1:  $v$ is non-archimedean and $D\otimes_LL_v$ is not split. 
Since $h_M$ is isotropic, $h_{M_j}$ is isotropic. 
By \cref{local2}, $h_{L_v}$ is isotropic. 

\noindent
Case 2: $v$ is non-archimedean and $D\otimes_LL_v$ is split. 
Then $D \otimes {M_j}$ is split. 
Since $h_M$ is isotropic, $h_{M_j}$ is isotropic. 
Suppose $h_{L_v}$ is Morita equivalent to a quadratic form $q$ over $L_v$. 
Then $h_{M_j}$ is Morita equivalent to the quadratic form $q_{M_j}$. 
Then $q_{M_j}$ is isotropic. 
By \cite{Spr1}, $q$ is isotropic over $L_v$. 
By Morita equivalence again, $h_{L_v}$ is isotropic. 

\noindent
Case 3: $v$ is archimedean.  
Any place $w\in \Omega_M$ that lies over $v$ is still archimedean. 
Since $[M_j:L_v]$ is odd, $M_j=L_v\simeq \mathbb R$ or $\mathbb C$. 
Since $h_M$ is isotropic, $h_{M_w}=h_{L_v}$ is isotropic. 

By three cases above, $h_{L_v}$ is isotropic for all $v\in\Omega_L$. 
Finally, by Landherr's local-global principle (see \cite{Land} or \cite[ch.~10, 6.2]{Sch}), $h$ is isotropic. 
\end{proof}


Finally, we revisit and prove \cref{odd} as \cref{odd2}. 

\begin{thm}\label{odd2}
Let $p$ be an odd prime. 
Let $K$ be a $p$-adic field. 
Let $F$ be the function field of a smooth, projective, geometrically integral curve over $K$.
Let $\Omega$ be the set of all rank one discrete valuations on $F$. 
Let $A$ be a finite-dimensional central simple $F$-algebra with an involution $\sigma$ of the first kind.  
Let $h: V\times V\to A$ be an $\varepsilon$-hermitian space over $(A, \sigma)$ for $\varepsilon\in\{1, -1\}$. 

Let $M$ be an odd degree extension of $F$. 
If $h_M$ is isotropic, then $h$ is isotropic. 
\end{thm}

\begin{proof}
In fact, by Morita equivalence \cite[ch.~I, 9.3.5]{Knus}, we assume that $A=D$ is a central division $F$-algebra. 
Suppose that $h_M$ is isotropic. 
Let $\deg D=d$, $\dim_D(V) = m$ and $i_W(h_M)$ the Witt index of $h_M$. 
Then $1\le i_W(h_M)\le \frac{m}{2}$ and $X_d(M) \ne \emptyset$, where $X_d$ is as in \cref{sec1.5}.  

Suppose $i_W(h_M)=\frac{m}{2}$. 
Then $h_M$ is hyperbolic. 
By \cite{BL}, $h$ is hyperbolic. 
 
Suppose that  $i_W(h_M)<\frac{m}{2}$.  
Let $v\in\Omega$.  
By \cref{ready}, we have an extension $M_j/F_v$ such that $[M_j: F_v]$ is odd. 
Let $k_j$ be the residue field of $M_j$ and $k(v)$ the residue field of $F_v$. 
Since $e(M_j/F_v)f(M_j/F_v)=[M_j: F_v]$ is odd, $[k_j: k(v)] =f(M_j/F_v)$ is odd. 
Since $X_d(M) \ne \emptyset$, we have $X_d(M \otimes F_v) \ne \emptyset$. 
In particular, $X_d(M_j) \ne \emptyset$. 
Since the residue fields are either local or global (see \cite[8.1]{P14}), $[k_j: k(v)]$ is odd and $\Per(D\otimes_F F_v)|2$, 
by \cref{local2} and \cref{global2}, the conditions in \cref{pre-odd} are satisfied. 
By Morita equivalence and \cref{pre-odd}, $X_{d}(F_v)\ne\emptyset$ for all $v$. 
Finally by the Hasse principle \cref{thm1*}, $X_{d}(F)\ne\emptyset$, so $h$ is isotropic. 
\end{proof}
 

\bibliographystyle{amsalpha} 
\bibliography{wu} 
\end{document}